\documentclass[11pt]{amsart}
\usepackage[utf8]{inputenc}
\usepackage{amsmath,amssymb,amsthm}
\usepackage{xcolor}
\usepackage[colorlinks=true,citecolor=blue,linkcolor=blue]{hyperref}
\usepackage[none]{hyphenat}

\usepackage[letterpaper, margin=2.5cm]{geometry}

\newtheorem{theorem}{Theorem}[section]
\newtheorem{proposition}[theorem]{Proposition}
\newtheorem{lemma}[theorem]{Lemma}
\newtheorem{corollary}[theorem]{Corollary}
\theoremstyle{definition}
\newtheorem{definition}[theorem]{Definition}
\theoremstyle{remark}
\newtheorem{remark}[theorem]{Remark}
\newtheorem{example}[theorem]{Example}

\newcommand{\M}{M}
\newcommand{\R}{\mathbb{R}}
\newcommand{\Hy}{\mathbb{H}}
\newcommand{\dg}{\Delta_g}
\newcommand{\Hess}{\operatorname{Hess}}

\title[Busemann profiles on Cartan--Hadamard manifolds]{Rigidity and existence
of Busemann profiles for the Allen--Cahn equation on Cartan--Hadamard manifolds}

\author{Luis Eduardo Osorio-Acevedo}
\author{\'Alvaro Jaramillo}
\address{Department of Mathematics \\
Technological University of Pereira \\
Barrio Álamos, Pereira, Risaralda, Colombia \\
Postal Code: 660003}
\email{leosorio@utp.edu.co}
\email{aljamil@utp.edu.co}

\begin{document}

\begin{abstract}
We study entire solutions of the Allen--Cahn equation
$\Delta_g u = W'(u)$ on Cartan--Hadamard manifolds of the form
$u = U\circ b$, where $b$ is a Busemann function and $U\in C^2(\mathbb{R})$;
we call these Busemann profiles. Our main result is a rigidity principle
requiring no homogeneity of the manifold: if $\inf \Delta_g b>0$ and $U$ has
finite limits $\ell_\pm$, then $W(\ell_-)\le W(\ell_+)$, with equality
exactly when $U$ is constant. In particular no nonconstant Busemann layer
joins the two wells of a balanced double-well potential, in any dimension
$n\ge 2$; under $K\le-\kappa_2<0$ the hypothesis is automatic by Hessian
comparison, and if $W'$ is locally Lipschitz it holds under the mere
nonvanishing of $\Delta_g b$. The mechanism is a dissipation identity for the
Hamiltonian $\tfrac12(U')^2-W(U)$ along the flow lines of $\nabla b$, in
which the mean curvature of the horospheres acts as a wave speed.
Conversely, when $\Delta_g b\equiv c$ and $W$ is bistable and unbalanced,
strictly increasing Busemann fronts joining the wells exist precisely when
$c$ is the Fife--McLeod speed; their interfaces are then horospheres of
constant mean curvature. On warped-line and rotational surfaces we test
sharpness: monotone warpings admit no equal-level profile and rotational
surfaces no radial two-well profile, while an even convex warping carries a
nonconstant monotone layer whose nodal set is the central geodesic leaf.
\end{abstract}

\maketitle

\section{Introduction}\label{sec1}

The De Giorgi conjecture asks whether every bounded entire solution of
$\Delta u=u^3-u$ on $\mathbb R^n$ that is strictly monotone in one
direction must be one-dimensional, at least for $n\le8$
\cite{DeGiorgi1979}. It was proved in its original form for $n=2$ by
Ghoussoub and Gui \cite{GhoussoubGui1998} and for $n=3$ by Ambrosio and
Cabr\'e \cite{AmbrosioCabre2000}, building on Liouville-type techniques of
\cite{BerestyckiCaffarelliNirenberg1997}. Savin proved it for $4\le n\le8$
under the additional assumption $\lim_{x_n\to\pm\infty}u(x',x_n)=\pm1$
\cite{Savin2009}; without this assumption these dimensions remain open. For
$n\ge9$ the conjecture is false \cite{delPinoKowalczykWei2011}. (The
asymptotic condition of \cite{Savin2009} is automatically satisfied by
travelling-wave-type profiles such as those studied below, where limits at
$\pm\infty$ are prescribed from the outset.)

More generally, the geometric relation between phase-transition layers and
minimal or constant-mean-curvature hypersurfaces on Riemannian manifolds has
been developed in several directions; see, for instance,
\cite{PacardRitore2003}. In negatively curved settings, hyperbolic space has
served as a natural testing ground for these phase-transition and interface
questions from several points of view. A complementary parabolic viewpoint
was developed by Matano, Punzo, and Tesei~\cite{MatanoPunzoTesei2015}, who
studied front propagation for KPP and Allen--Cahn nonlinearities on
$\Hy^n$. They introduced \emph{horospheric waves}, whose level sets are
moving horospheres. In the upper half-space model these solutions have the
form
\[
u(x,t)=q^*(-\log x_n-c^*t), \qquad c^*=c_0-(n-1),
\]
where $c_0$ is the one-dimensional front speed. Thus the constant
horospherical mean-curvature term $n-1$ appears explicitly as a geometric
drift modifying the propagation speed. Pisante and Ponsiglione studied
phase-transition minimizers and their relation to the asymptotic Plateau
problem in $\Hy^n$ \cite{PisantePonsiglione2011}, while Mazzeo and
S\'aez constructed multiple-layer solutions with nodal set close to a
prescribed, widely separated collection of hyperplanes \cite{MS2014}.

On hyperbolic space this was taken up by Birindelli and Mazzeo
\cite{BirindelliMazzeo2009}, who studied solutions invariant under the three
cohomogeneity-one subgroups of $\operatorname{Isom}(\Hy^n)$ and, in each case,
analyzed existence or nonexistence of layers joining the wells; separately, under
decay hypotheses, they showed that solutions with invariant asymptotic data are
themselves invariant. Their arguments use the isometry group throughout: the
reduced equation is autonomous or explicitly computable precisely because the
ambient space is homogeneous.

This paper asks what survives when homogeneity is removed and the curvature is
not necessarily constant. Let $(\mathcal M^n,g)$ be a Cartan--Hadamard
manifold, let $\xi\in\partial_\infty\mathcal M$ be an ideal boundary point,
and let $b_\xi$ be the associated Busemann function. The horospheres at
$\xi$ are the level sets of $b_\xi$, and $\Delta_g b_\xi$ is their mean
curvature. We study solutions of the form $u=U\circ b_\xi$, which we call
\emph{Busemann profiles}: functions constant on each horosphere. On $\Hy^n$
these are exactly the solutions invariant under a parabolic subgroup, but the
definition needs no isometries, only the Busemann function itself.

The reduction is elementary but, in the absence of horospherical
homogeneity, is naturally pointwise. Since $|\nabla b_\xi|\equiv1$,
\[
\Delta_g(U\circ b_\xi)(x)
=
U''(b_\xi(x)) + U'(b_\xi(x))\,\Delta_g b_\xi(x).
\]
Hence a profile $u=U\circ b_\xi$ solves $\Delta_g u=W'(u)$ if and only if
\[
U''(b_\xi(x)) + U'(b_\xi(x))\,\Delta_g b_\xi(x)
=
W'\bigl(U(b_\xi(x))\bigr)
\qquad\text{for every }x\in\mathcal M.
\]
Write $\Sigma_t:=b_\xi^{-1}(t)$ for the horospheres at $\xi$. For
$p\in\Sigma_0$, let $\varphi_p$ be the integral curve of
$\nabla b_\xi$ through $p$, so that $b_\xi(\varphi_p(t))=t$. Along this
curve the equation becomes
\[
U''(t) + \alpha_p(t)\,U'(t) = W'(U(t)), \qquad
\alpha_p(t):=\Delta_g b_\xi(\varphi_p(t)).
\]
Different flow lines may in general produce different friction coefficients
$\alpha_p$. If $\Delta_g b_\xi$ is constant on each horosphere,
$\alpha_p$ is independent of $p$ and the reduction becomes a single
nonautonomous ODE; in the constant-drift case $\Delta_g b_\xi\equiv c$, it
reduces further to the classical travelling-wave equation with constant
speed $c$. Conversely, a nonconstant profile forces such leaf-constancy
\emph{a posteriori} wherever $U'\neq0$ (Remark~\ref{rem:consistency}): the
ansatz is self-limiting rather than automatically reducible.

Our main result is a rigidity statement that requires no horospherical
constancy of $\Delta_g b_\xi$, only the uniform positive lower bound
\[
c:=\inf_{\mathcal M}\Delta_g b_\xi>0.
\]
Along every flow line this yields $\alpha_p(t)\ge c$, which is all that the
dissipation argument needs.

The mechanism is the dissipation identity for the Hamiltonian
$H=\tfrac12(U')^2-W(U)$. Along any chosen flow line one has
\[
H'=-\alpha_p(t)(U')^2.
\]
The proof uses only the uniform estimate $\alpha_p(t)\ge c>0$, inherited
from $\inf_{\mathcal M}\Delta_g b_\xi>0$; no constancy of the coefficient
is assumed as a hypothesis of the theorem.

The converse direction is equally geometric. When $\Delta_g b_\xi\equiv c$ is
constant and $W$ is a bistable potential satisfying the hypotheses of
Proposition~\ref{prop:compatibility-front}, with $W(-1)<W(1)$, a strictly
increasing Busemann front joining the wells exists exactly when $c$ equals
the Fife--McLeod speed $\tau^*(W')$ (Proposition~\ref{prop:compatibility-front});
the interfaces are then exact CMC horospheres. The constant-drift geometric
hypothesis is not vacuous: it holds on hyperbolic space, on harmonic
manifolds, and on the nonsymmetric Damek--Ricci spaces. Between the two
extremes lies the heterogeneous regime, where $\Delta_g b_\xi$ depends on
$b_\xi$ alone; Proposition~\ref{prop:dictionary} records the resulting
dictionary between variable-curvature geometry and nonautonomous friction.

Section~\ref{sec:warped} tests the sharpness of these results on surfaces,
where warped metrics $g=dt^2+f(t)^2\,ds^2$ make everything explicit. There
the uniform drift hypothesis can be weakened to $a=f'/f\ge0$ with
$a\not\equiv0$ (Theorem~\ref{thm:translational}), though within that class
$t$ need not be a Busemann function, so this does not by itself extend
Theorem~\ref{thm:main} within the Busemann class; rotational surfaces admit
no radial two-well profiles, for a different reason, namely the boundary
condition at the pole (Theorem~\ref{thm:radial}); and a reflection-symmetric
geodesic leaf carries an odd monotone layer (Theorem~\ref{thm:existence}).
The last of these lies outside the Busemann-profile setting of
Theorem~\ref{thm:main}: there the coordinate $t$ is signed distance to a
geodesic, not a Busemann function, and $a(t)=\tanh t$ changes sign, impossible
for a genuine Busemann function, whose Laplacian is always nonnegative
(Proposition~\ref{prop:flow-geodesic}). The example shows that sign-changing
drift can support nonconstant equal-level layers in this broader warped
setting, but it is not itself a sharpness example within the Busemann class.

Some nonvanishing of the drift is, however, indispensable even within that
class: in $\mathbb R^n$, affine Busemann functions satisfy $\Delta b\equiv0$,
and the classical balanced heteroclinic yields a nonconstant profile
$u=U\circ b$. That degenerate case turns out to be the only obstruction. If
$W'$ is locally Lipschitz, in particular if $W\in C^2$, the same dissipation
argument combined with Cauchy uniqueness shows that the strict lower bound
$c>0$ may be relaxed all the way to $\Delta b\not\equiv0$, which already
forces an equal-level Busemann profile to be constant
(Corollary~\ref{cor:nonnegative-drift}); and $\Delta b\equiv0$ occurs exactly
when $\mathcal M$ splits off a flat line (Remark~\ref{rem:dichotomy}). What
remains open is the borderline $C^1$ regularity of Theorem~\ref{thm:main},
where Cauchy uniqueness may fail.

The paper is organized as follows. Section~\ref{sec:busemann} collects the
properties of Busemann functions we use: regularity, the gradient identity, the
geodesic character of the flow, the second fundamental form of horospheres, and
the invariant reduction $\Delta(U\circ b)=U''+U'\,\Delta b$
(Lemma~\ref{lem:invariant}). We also include a proof of the comparison bounds
for $\Delta b$ by a distributional limit argument. This material is standard,
but we have recorded it in the sign convention used throughout, since the
direction in which $\Delta_g b_\xi$ is positive is exactly what the main
theorem is about. Section~\ref{sec:main} proves the rigidity theorem and, in
the constant-drift, unbalanced regime, an exact-front compatibility criterion
for strictly increasing Busemann fronts. Section~\ref{sec:warped} specializes
to surfaces, where the nonexistence hypotheses can be relaxed and the layer
across a symmetric geodesic leaf is constructed by shooting, and its
stability is discussed. Section~\ref{sec:remarks} discusses the energy of
these layers and open problems.

Two limitations should be stated at the outset. The results concern the
profile class $u=U\circ b_\xi$ and do not classify arbitrary bounded entire
solutions; removing the ansatz, as the symmetry theorems of
\cite{BirindelliMazzeo2009} do on $\Hy^n$, would require tools that the
absence of an isometry group leaves unavailable. A recent preprint of
Cavalcante, Espinar, and Mar\'in \cite{CavalcanteEspinarMarin2026a} takes a
complementary route on this point: working directly with the asymptotic
Dirichlet problem on Cartan--Hadamard manifolds, and without any invariance
ansatz, they obtain non-existence of bounded solutions to $\Delta u+f(u)=0$
with prescribed asymptotic boundary data, via convex barrier hypersurfaces
that substitute for the totally geodesic foliations available on $\Hy^n$; it
remains open how their approach interacts with the profile-based rigidity
obtained here. Moreover, the existence result covers warped-line surfaces
with even warping; the case of a convex warping with an isolated interior
minimum but without reflection symmetry about the corresponding geodesic
leaf is not settled here, so we claim no criterion characterizing when
layers exist.

\section{Busemann functions on Cartan--Hadamard manifolds}
\label{sec:busemann}

Throughout this work, let $(\mathcal M^n, g)$, $n\ge 2$, denote a \emph{smooth
Cartan--Hadamard manifold}, i.e., a complete, simply connected Riemannian
manifold with nonpositive sectional curvature ($K\le 0$). By the
Cartan--Hadamard theorem, the exponential map $\exp_p : T_p\mathcal M \to
\mathcal M$ is a diffeomorphism for every $p\in\mathcal M$; consequently,
$\mathcal M$ is diffeomorphic to $\mathbb R^n$.

Fix a base point $o\in\mathcal M$. For each ideal boundary point $\xi \in
\partial_\infty \mathcal M$, let
\[
\gamma_\xi : [0,\infty) \to \mathcal M, \qquad \gamma_\xi(0)=o,
\]
be the unique unit-speed geodesic ray emanating from $o$ with asymptotic
endpoint $\xi$. The \emph{Busemann function} associated to $\xi$ (and
normalized by $o$) is defined by the pointwise limit
\begin{equation}
\label{eq:busemann_def_revised}
b_\xi(x) \;:=\; \lim_{t\to\infty} \big( d(x, \gamma_\xi(t)) - t \big), \qquad x\in\mathcal M.
\end{equation}

The existence of the limit is a direct consequence of the triangle
inequality. Indeed, for $0\le s \le t$, we have
\[
d(x,\gamma_\xi(t)) - t \;\le\; d(x,\gamma_\xi(s)) + d(\gamma_\xi(s),\gamma_\xi(t)) - t
= d(x,\gamma_\xi(s)) - s,
\]
so the function $t\mapsto d(x,\gamma_\xi(t))-t$ is nonincreasing. Moreover, by
the reverse triangle inequality,
\[
d(x,\gamma_\xi(t)) \;\ge\; d(\gamma_\xi(0),\gamma_\xi(t)) - d(x,\gamma_\xi(0))
= t - d(x,o),
\]
hence $d(x,\gamma_\xi(t))-t \ge -d(x,o)$. Thus the limit in
\eqref{eq:busemann_def_revised} is finite for every $x$.

A direct evaluation along the defining ray yields the crucial normalization
\begin{equation}
\label{eq:busemann_sign_revised}
b_\xi(\gamma_\xi(s)) = -s \qquad \text{for all } s\ge 0.
\end{equation}

We adopt Eberlein's sign convention \cite[1.10.1]{Eberlein1996}, so that
$b_\xi(\gamma_\xi(t))=-t$ along geodesic rays; with this convention $b_\xi$ is
convex, and hence $\mathrm{Hess}\,b_\xi\ge0$, the sign used throughout the
comparison arguments below. With this convention, the gradient of $b_\xi$ is
given by
\begin{equation}
\label{eq:grad_busemann_revised}
\nabla b_\xi (x) = -\dot{\gamma}_{x\to \xi}(0),
\end{equation}
where $\gamma_{x\to\xi}$ denotes the geodesic ray from $x$ to $\xi$.
In particular, $b_\xi$ is a $C^2$-function (cf. \cite[Prop.~3.1]{HeintzeImHof1977} or
\cite[1.10.2]{Eberlein1996}). The horospheres $\Sigma_t:=b_\xi^{-1}(t)$
are thus $C^2$ convex hypersurfaces, with the non-negative Hessian ensuring
the convexity of the horoballs $B_\xi(t):=b_\xi^{-1}\bigl((-\infty,t]\bigr)$.

\begin{lemma}[Convexity and Lipschitz bound]\label{lem:convexity-lipschitz}
The Busemann function $b_\xi$ is convex and $1$-Lipschitz.
\end{lemma}
\begin{proof}
Set $f_t(x):=d(x,\gamma_\xi(t))-t$, so that $f_t \downarrow b_\xi$ pointwise
by the monotonicity established above.

On a Cartan--Hadamard manifold the distance function $d(\cdot,q)$ is convex
for every $q$, hence so is each $f_t$. Let $\gamma_{xy}:[0,1]\to\mathcal M$
be the geodesic with $\gamma_{xy}(0)=x$, $\gamma_{xy}(1)=y$. For
$\lambda\in[0,1]$,
\[
f_t(\gamma_{xy}(\lambda)) \le (1-\lambda) f_t(x)+\lambda f_t(y),
\]
and letting $t\to\infty$ gives
\[
b_\xi(\gamma_{xy}(\lambda)) \le (1-\lambda) b_\xi(x)+\lambda b_\xi(y).
\]
To see that it is Lipschitz, the reverse triangle inequality gives
$|f_t(x)-f_t(y)|\le d(x,y)$ for every $t\ge 0$.
Letting $t\to\infty$ yields $|b_\xi(x)-b_\xi(y)|\le d(x,y)$.
\end{proof}

\begin{proposition}[Regularity and gradient]\label{prop:gradient-norm}
The Busemann function $b_\xi$ is of class $C^2$ on $\mathcal M$.\footnote{This
is stated for a general Cartan--Hadamard manifold in
\cite[Prop.~3.1]{HeintzeImHof1977} (attributed there to Eberlein), whose
proof invokes a lower curvature bound only on a compact neighbourhood, where
it is automatic. Some secondary sources record this regularity under pinched
negative curvature instead; that hypothesis belongs to the stronger
statement $Z\in C^2$ (i.e.\ $b_\xi\in C^3$), which we do not use.} Moreover,
if $\sigma_p:[0,\infty)\to\mathcal M$ denotes the unique unit-speed geodesic
ray starting at $p$ that is asymptotic to $\gamma_\xi$, then
\[
\nabla b_\xi(p)=-\sigma_p'(0).
\]
In particular, $|\nabla b_\xi(p)|=1$ for every $p\in\mathcal M$.
\end{proposition}
\begin{proof}
The $C^2$-regularity is classical; by \cite[Prop.~3.1]{HeintzeImHof1977}, the
radial field associated with an ideal point is $C^1$ and equals the
negative gradient of the corresponding Busemann function. Hence
$b_\xi\in C^2(\mathcal M)$ (see also \cite[1.10.2(1)]{Eberlein1996}).

Fix $p\in\mathcal M$. There is a unique unit-speed ray
$\sigma_p:[0,\infty)\to\mathcal M$ with $\sigma_p(0)=p$ asymptotic to
$\gamma_\xi$, i.e.\ $\sup_{t\ge 0} d(\gamma_\xi(t),\sigma_p(t))<\infty$
\cite[Prop.~II.8.2]{BridsonHaefliger1999}; equivalently, $\sigma_p$
represents the same ideal point $\xi$.

Let $b_\xi^{(p)}$ be the Busemann function built from $\sigma_p$. By
\eqref{eq:busemann_sign_revised} applied to $\sigma_p$,
\[
b_\xi^{(p)}(\sigma_p(s)) = -s \qquad (s\ge 0).
\]
Since $b_\xi$ and $b_\xi^{(p)}$ are Busemann functions at the same ideal
point, they differ by a constant \cite[1.10.2(3)]{Eberlein1996}; evaluating
at $p$ gives $b_\xi^{(p)} = b_\xi - b_\xi(p)$, hence
\[
b_\xi(\sigma_p(s)) = b_\xi(p) - s \qquad (s\ge 0).
\]
As $b_\xi$ is $C^1$, the right derivative at $s=0$ computes the
differential:
\[
\langle \nabla b_\xi(p),\, v\rangle = -1, \qquad v:=\sigma_p'(0),\ \ |v|=1 .
\]
On the other hand $b_\xi$ is $1$-Lipschitz
(Lemma~\ref{lem:convexity-lipschitz}), so $|\nabla b_\xi(p)|\le 1$.
Cauchy--Schwarz then gives
\[
1 = |\langle \nabla b_\xi(p), v\rangle| \le |\nabla b_\xi(p)|\,|v| \le 1,
\]
so equality holds throughout: $|\nabla b_\xi(p)|=1$ and the equality case
forces $\nabla b_\xi(p) = -v = -\sigma_p'(0)$. Since $p$ was arbitrary, the
proof is complete.
\end{proof}

All Busemann functions in the sequel are normalized at the fixed base point
$o$, so that $b_\xi(o)=0$ for every $\xi\in\partial_\infty\mathcal M$.

In the following, whenever the center $\xi$ is fixed and no confusion can
arise, we simply write $b:=b_\xi$.

\begin{proposition}[Geodesic gradient flow]\label{prop:flow-geodesic}
The following identities hold pointwise on $\mathcal M$,
\[
\Hess b(\nabla b,\cdot)\equiv 0,
\qquad
\Delta b = \sum_{i=2}^n \Hess b(e_i,e_i) \ge 0
\]
for any orthonormal frame $\{e_1=\nabla b,e_2,\dots,e_n\}$. In particular,
the integral curves of $\nabla b$ are unit-speed geodesics:
\[
\nabla_{\nabla b}\nabla b = 0.
\]
\end{proposition}
\begin{proof}
Differentiating $|\nabla b|^2\equiv 1$ along an arbitrary vector field $X$
gives
\[
0 = X|\nabla b|^2 = 2\langle \nabla_X\nabla b,\nabla b\rangle
= 2\,\Hess b(X,\nabla b),
\]
and by symmetry of the Hessian, $\Hess b(\nabla b,\cdot)\equiv 0$. Since
$\Hess b(\nabla b,X)=\langle \nabla_{\nabla b}\nabla b, X\rangle$ for all
$X$, non-degeneracy of $g$ yields $\nabla_{\nabla b}\nabla b=0$; thus the
integral curves of $\nabla b$ are geodesics parametrized by arc length.

In an orthonormal frame $\{e_1=\nabla b,e_2,\dots,e_n\}$ the first term of
the trace vanishes, so
\[
\Delta b = \sum_{i=1}^n \Hess b(e_i,e_i) = \sum_{i=2}^n \Hess b(e_i,e_i).
\]
Convexity of $b$ (Lemma~\ref{lem:convexity-lipschitz}) together with its
$C^2$-regularity gives $\Hess b\ge 0$ as a bilinear form, whence its trace
$\Delta b\ge 0$.
\end{proof}

\begin{corollary}[Second fundamental form of horospheres]\label{cor:horosphere}
Let $\Sigma=\{b=t\}$ be a horosphere, equipped with the outward unit normal
$N:=\nabla b$ (outward with respect to the horoball
$B_\xi(t)=b^{-1}\bigl((-\infty,t]\bigr)$), and define
\[
\mathrm{II}(X,Y):=\langle \nabla_X N, Y\rangle,\qquad X,Y\in T\Sigma.
\]
Then:
\begin{enumerate}
\item $\mathrm{II} = \Hess b|_{T\Sigma}$, and the mean curvature of $\Sigma$
with respect to $N$ is
\[
H_\Sigma = \operatorname{tr}_\Sigma \mathrm{II} = \Delta b \;\ge\; 0 .
\]
\item For any unit $T\in T\Sigma$, the normal curvature in the direction
$T$, defined by $k_g(T):=\langle \nabla_T T, N\rangle$, satisfies
\[
k_g(T) = -\Hess b(T,T).
\]
\item If $n=2$, then $T$ spans $T\Sigma$ and $k_g=-\Delta b$.
\end{enumerate}
The sign discrepancy between \textup{(1)} and \textup{(2)} is only a matter
of convention: $k_g(T)=-\mathrm{II}(T,T)$ by definition. For $n\ge 3$,
$k_g(T)$ depends on $T$ and is not determined by $\Delta b$ alone.
\end{corollary}
\begin{proof}
(1) The identity $\mathrm{II}=\Hess b|_{T\Sigma}$ is immediate from
$\mathrm{II}(X,Y)=\langle\nabla_X\nabla b,Y\rangle=\Hess b(X,Y)$. Since
$\Hess b(\nabla b,\cdot)=0$ by Proposition~\ref{prop:flow-geodesic}, the
trace over $T\Sigma$ agrees with the full trace, so $H_\Sigma=\Delta b$;
non-negativity is Proposition~\ref{prop:flow-geodesic}.

(2) Let $T$ be a unit tangent field to $\Sigma$. Differentiating
$\langle T,N\rangle=0$ along $T$ gives
\[
\langle \nabla_T T,N\rangle = -\langle T,\nabla_T N\rangle = -\Hess b(T,T).
\]

(3) For $n=2$ the adapted frame is $\{e_1=N,e_2=T\}$, so
$\Delta b=\Hess b(T,T)$ and the claim follows from (2).
\end{proof}

\begin{remark}[Orientation and the affine case]\label{rem:orientation}
With our sign convention every Busemann function is convex, hence
$\Delta b_\xi\ge 0$ pointwise. Suppose both $b_\xi$ and $-b_\xi$ were
Busemann functions. Convexity of the former gives $\Hess b_\xi\ge 0$ and
convexity of the latter gives $\Hess b_\xi\le 0$, so $\Hess b_\xi\equiv 0$
and $b_\xi$ is affine; equivalently, by Corollary~\ref{cor:horosphere}, all
horospheres at $\xi$ are totally geodesic.

In $\mathbb R^n$ this happens for every Busemann function:
$b_v(x)=-\langle x,v\rangle$ and $-b_v=b_{-v}$. In a general
Cartan--Hadamard manifold, however, $\Hess b_\xi\equiv 0$ does not force
flatness. Indeed, it says precisely that $\nabla b_\xi$ is a parallel unit
vector field, so by the de Rham decomposition theorem $\mathcal M$ splits
isometrically as
\[
\mathcal M \cong \mathbb R \times \mathcal N,
\]
with $b_\xi$ equal, up to an additive constant, to minus the projection onto
the first factor. The factor $\mathcal N$ is again Cartan--Hadamard but need
not be flat.

By contrast, in $\mathbb H^n$ every Busemann function satisfies
$\Delta b_\xi\equiv n-1$, so $-b_\xi$ is never a Busemann function there
(see Example~\ref{ex:H2}).
\end{remark}

We recall the following elementary but fundamental identity, the only tool
from this section needed in the proof of Theorem~\ref{thm:main}. Its virtue
is that it requires no coordinates and holds in any dimension.

\begin{lemma}[Invariant reduction identity]\label{lem:invariant}
Let $U\in C^2(\mathbb R)$. Then $u:=U\circ b\in C^2(\mathcal M)$ and,
pointwise on $\mathcal M$,
\begin{equation}\label{eq:invariant}
  \Delta u \;=\; U''(b)\;+\;U'(b)\,\Delta b .
\end{equation}
\end{lemma}
\begin{proof}
Since $b\in C^2(\mathcal M)$ and $U\in C^2(\mathbb R)$, $u$ is $C^2$, with
$\nabla u=U'(b)\nabla b$. Hence
\[
\Delta u=\operatorname{div}\!\big(U'(b)\nabla b\big)
= U''(b)|\nabla b|^2+U'(b)\,\Delta b
= U''(b)+U'(b)\,\Delta b,
\]
since $|\nabla b|\equiv 1$.
\end{proof}

For completeness we record the global structure underlying the ``Busemann
coordinates'' picture; the proof of the main theorem does not use it beyond
a single flow line.

\begin{proposition}[Global $C^1$ decomposition]\label{prop:global}
Let $\Sigma_0 := b_\xi^{-1}(0)$, and for each $x\in\mathcal M$, let
$\gamma_{x\to \xi} : [0,\infty) \to \mathcal M$ denote the unique unit-speed
geodesic ray from $x$ to the ideal point $\xi$, so that
$\gamma_{x\to \xi}(0)=x$ and $\dot{\gamma}_{x\to \xi}(0)$ points toward
$\xi$. Then:
\begin{enumerate}
  \item[\textup{(i)}] every level set $\Sigma_t := b^{-1}(t)$ is a nonempty,
  connected, properly embedded $C^2$-hypersurface;
  \item[\textup{(ii)}] the vector field $\nabla b_\xi$ is complete.
  Denoting its flow by $\Phi:\mathbb R\times\mathcal M\to\mathcal M$,
  $(t,x)\mapsto\Phi_t(x)$, each curve $t\mapsto\Phi_t(x)$ is a unit-speed
  geodesic and
\[
  b(\Phi_t(x))=b(x)+t \qquad\text{for all }(t,x)\in\mathbb R\times\mathcal M;
\]
moreover, for $t\le0$,
\[
  \Phi_t(x)=\gamma_{x\to\xi}(-t),
\]
so the flow line $t\mapsto\Phi_t(x)$ extends the ray $\gamma_{x\to\xi}$ to a
complete geodesic through $x$;
  \item[\textup{(iii)}] the restriction $\Phi|_{\mathbb R \times \Sigma_0}$
  is a $C^1$-diffeomorphism. Moreover, the leaves $\Sigma_t$ are orthogonal
  to the flow lines, and $\Phi^*g = dt^2 + h_t$, where $\{h_t\}$ is a
  continuous family of Riemannian metrics on $\Sigma_0$ (with $h_t$
  denoting the pullback of the induced metric on $\Sigma_t$).
\end{enumerate}
If $n=2$, then $\Sigma_0$ is diffeomorphic to $\mathbb R$ and has infinite
length.
\end{proposition}
\begin{proof}
(ii) By Proposition~\ref{prop:gradient-norm}, $\nabla b$ is a $C^1$ unit
vector field, and by Proposition~\ref{prop:flow-geodesic} its integral
curves are unit-speed geodesics. Since $(\mathcal M,g)$ is complete, it is
geodesically complete by the Hopf--Rinow theorem, so these integral curves
are defined for all $t\in\mathbb R$; thus $\nabla b$ is a complete vector
field, generating a global $C^1$ flow $\Phi_t$. Along any flow line,
\[
\frac{d}{dt}b(\Phi_t(x))
=\langle\nabla b(\Phi_t(x)),\dot\Phi_t(x)\rangle
=|\nabla b(\Phi_t(x))|^2=1,
\]
hence $b(\Phi_t(x))=b(x)+t$.

To identify $\Phi_t$ with $\gamma_{x\to\xi}$ for $t\le0$: both
$s\mapsto\Phi_{-s}(x)$ ($s\ge0$) and $s\mapsto\gamma_{x\to\xi}(s)$ are
unit-speed geodesics issued from $x$ with the same initial velocity, since
$\dot\Phi_{-s}(x)\big|_{s=0}=-\nabla b(x)=\dot\gamma_{x\to\xi}(0)$ by
Proposition~\ref{prop:gradient-norm}. By uniqueness of geodesics with given
initial point and velocity, the two curves coincide on their common domain
$[0,\infty)$; equivalently $\Phi_t(x)=\gamma_{x\to\xi}(-t)$ for $t\le0$.

(i) Since $|\nabla b| \equiv 1$, $b$ has no critical points. Hence each
$\Sigma_t = b^{-1}(t)$ is a $C^2$-embedded hypersurface; being closed in
$\mathcal M$, it is properly embedded. It is nonempty because
$t\mapsto b(\Phi_t(x))$ is onto $\mathbb R$ along each flow line. For
connectedness, define $r_t : \mathcal M \to \Sigma_t$ by
$r_t(x) := \Phi_{t-b(x)}(x)$. By (ii), this is continuous, and restricts to
the identity on $\Sigma_t$. Hence $\Sigma_t$ is the continuous image of the
connected space $\mathcal M$, so it is connected.

(iii) \emph{Injectivity:} if $\Phi_t(p)=\Phi_{t'}(p')$ with
$p,p'\in\Sigma_0$, applying $b$ gives $t=t'$; since each flow map $\Phi_t$
is a diffeomorphism with inverse $\Phi_{-t}$, it follows that $p=p'$.
\emph{Surjectivity:} $x=\Phi_{b(x)}(r_0(x))$.
\emph{Differential:} $d\Phi(\partial_t)=\nabla b$; for $v\in T_p\Sigma_0$,
differentiating $b(\Phi_t(q))=b(q)+t$ in $q$ gives
$db(d\Phi_t(v))=db(v)=0$, so $d\Phi_t(v)\perp\nabla b$, and
$d\Phi_t(v)\neq0$ since $\Phi_t$ is a diffeomorphism. Hence $d\Phi$ is
everywhere nonsingular and $\Phi$, being bijective, is a global
$C^1$-diffeomorphism. The same computation gives the orthogonality of the
leaves and, together with $g(\nabla b,\nabla b)=1$, the form
$\Phi^*g=dt^2+h_t$; since $\Phi$ is only $C^1$, the coefficients of $h_t$
are continuous in $t$.

Finally let $n=2$. A connected boundaryless $1$-manifold is diffeomorphic to
$\mathbb R$ or $S^1$. If $\Sigma_0$ were a circle it would bound a compact
disk $D\subset\mathcal M\cong\mathbb R^2$ (Jordan--Schoenflies), on which
the continuous function $b$ attains its extrema. Neither can be attained at
an interior point, where $\nabla b=0$ is impossible; hence both are
attained on $\partial D$, where $b\equiv 0$, forcing $b\equiv 0$ on $D$ and
again $\nabla b=0$ in the interior, a contradiction. Thus $\Sigma_0\cong
\mathbb R$. Its length is infinite: otherwise its arclength parametrization
would be defined on a bounded interval $(0,L)$ and be Cauchy as
$s\to L^-$, hence convergent by completeness, producing a limit point of
the end and contradicting properness.
\end{proof}

\begin{remark}[Coordinate form; regularity caveat]\label{rem:coordinates}
Along flow lines starting on $\Sigma_0$ the flow parameter $t$ of
Proposition~\ref{prop:global} coincides with $b$, and we write $b$ for it
below. When $b$ happens to be of class $C^3$ (in particular smooth, as on
$\mathbb H^n$ or in the models of Section~\ref{sec:warped}, where the
analogous coordinate is smooth \emph{by construction}), one may write
$\Phi^*g=db^2+h_b$ with a $b$-differentiable family of leaf metrics and
\[
  \Delta b = \partial_b\log\sqrt{\det h_b}\,;
\]
on surfaces, parametrizing $\Sigma_0$ by arclength $s$, one has
$h_b=m(b,s)^2\,ds^2$ with the \emph{horocyclic volume factor} $m$, and one
recovers the familiar coordinate form of the Laplacian
\[
\Delta u = u_{bb} + (\Delta b)\,u_b + \Delta_{\Sigma_b}u.
\]
For a general Cartan--Hadamard manifold, where only $b\in C^2$ is
guaranteed, we do not assert this coordinate form: by
Proposition~\ref{prop:global}(iii) the family $h_b$ is merely continuous in
$b$. All proofs below rely instead on Lemma~\ref{lem:invariant} and on
single flow lines, for which $C^2$ regularity suffices.
\end{remark}

\begin{example}[The upper half-space model]\label{ex:H2}
Let $\Hy^n=\{(x,y): x\in\mathbb R^{n-1},\,y>0\}$ with
$g=(|dx|^2+dy^2)/y^2$, of constant curvature $K\equiv-1$, and let
$\gamma(t)=(0,e^t)$ be the geodesic ray towards the ideal point $\infty$.
For a general point $(x,y)\in\Hy^n$, the hyperbolic distance to
$\gamma(t)$ is given by
\[
\cosh d\big((x,y),(0,e^t)\big) = 1 + \frac{|x|^2+(e^t-y)^2}{2ye^t},
\]
and $\operatorname{arccosh}(z)=\log(2z)+o(1)$ as $z\to\infty$, together with
\[
2\left(1+\frac{|x|^2+(e^t-y)^2}{2ye^t}\right)=\frac{|x|^2+y^2+e^{2t}}{ye^t},
\]
we get
\[
d\big((x,y),(0,e^t)\big)-t
=\log\!\left(\frac{|x|^2+y^2+e^{2t}}{y\,e^{2t}}\right)+o(1)
\;\xrightarrow[t\to\infty]{}\;\log\frac1y .
\]
Hence
\[
b(x,y) = -\log y,
\]
consistent with the normalization $b(\gamma(t))=-t$.

Setting $s:=x$, the metric takes the warped product form
\[
g = db^2 + e^{2b}\,|ds|^2 ,
\]
so $|\nabla b|\equiv 1$, the level sets $\{b=\text{const}\}$ are the
Euclidean horizontal hyperplanes, and by Remark~\ref{rem:coordinates},
\[
\Delta b = \partial_b\log\sqrt{\det h_b}
= \partial_b\log e^{(n-1)b} = n-1 .
\]

Since $\operatorname{Isom}(\Hy^n)$ acts transitively on
$\partial_\infty\Hy^n$, this holds for every ideal point: if $\phi$ maps
$\infty$ to $\xi$, then $b_\xi = b_\infty\circ\phi^{-1}+C_\phi$ for some
constant $C_\phi$, and since the Laplacian is isometry-invariant and
annihilates constants,
\[
\Delta b_\xi = (\Delta b_\infty)\circ\phi^{-1} = n-1 .
\]
In particular $\Delta b_\xi>0$ everywhere, so $-b_\xi$ is never a Busemann
function on $\Hy^n$, as asserted in Remark~\ref{rem:orientation}.
\end{example}

Under a strict upper curvature bound, $\Delta b$ is uniformly positive.
This is the only comparison input needed for Corollary~\ref{cor:curvature}.
The corresponding horosphere comparison is classical; see, for instance,
\cite{HeintzeImHof1977}. For completeness, we give the short
distributional-limit argument that passes from the standard Hessian
comparison estimates for the approximating distance functions
$d(\cdot,\gamma(t))$ to the corresponding bounds for $b$. This passage
requires only the local uniform convergence of $d(\cdot,\gamma(t))-t$,
which follows from the equi-Lipschitz property.

\begin{lemma}[Local uniform convergence from pointwise convergence]
\label{lem:equi-lipschitz-conv}
Let $(X,d)$ be a metric space and let $\{f_t\}_{t\ge 0}$ be an
equi-Lipschitz family of functions on $X$ converging pointwise to a
function $f$. Then $f$ is Lipschitz with the same constant, and the
convergence is uniform on every compact subset of $X$.
\end{lemma}
\begin{proof}
The Lipschitz bound passes to the pointwise limit. Given a compact
$K\subset X$ and $\varepsilon>0$, let $L\ge0$ be a common Lipschitz
constant and set $\delta:=\varepsilon/(3(1+L))$. Choose a finite
$\delta$-net $\{x_i\}$ in $K$. By pointwise convergence at the finitely
many $x_i$, there is $T$ such that $|f_t(x_i)-f(x_i)|<\varepsilon/3$ for
all $t\ge T$ and all $i$. For any $x\in K$, choose $i$ with
$d(x,x_i)<\delta$. Then
\[
|f_t(x)-f(x)|\le |f_t(x)-f_t(x_i)|+|f_t(x_i)-f(x_i)|+|f(x_i)-f(x)|
\le 2L\delta+\frac{\varepsilon}{3}<\varepsilon .
\]
Hence $f_t\to f$ uniformly on $K$.
\end{proof}

\begin{lemma}[One-sided comparison]\label{lem:comparison}
Assume the sectional curvature satisfies $K\le-\kappa_2<0$.
Then
\[
  \Delta b\;\ge\;(n-1)\sqrt{\kappa_2}\qquad\text{on }\mathcal M .
\]
If in addition $K\ge-\kappa_1$ with $\kappa_1\ge\kappa_2$, then also
$\Delta b\le(n-1)\sqrt{\kappa_1}$.
\end{lemma}
\begin{proof}
Write $b_t:=d(\cdot,\gamma(t))-t$ and $c_0:=(n-1)\sqrt{\kappa_2}$. Fix
$\varphi\in C_c^\infty(\mathcal M)$ with $\varphi\ge0$, and set
$\mathcal K:=\operatorname{supp}\varphi$. Since $d(\gamma(0),\gamma(t))=t\to
\infty$, for all $t$ large we have $\gamma(t)\notin\mathcal K$; as
$\mathcal M$ is Cartan--Hadamard and hence has no cut locus,
$r_t:=d(\cdot,\gamma(t))$ is smooth on a neighbourhood of $\mathcal K$. The
Hessian comparison theorem under $K\le-\kappa_2$
\cite[Thm.~6.4.3]{Petersen2016} gives
\[
  \Delta r_t\;\ge\;(n-1)\sqrt{\kappa_2}\,
  \coth\!\big(\sqrt{\kappa_2}\,r_t\big)\;\ge\;c_0
  \qquad\text{on }\mathcal K .
\]
Since $b_t=r_t-t$ we have $\Delta b_t=\Delta r_t$ there, and as $\Delta$ is
formally self-adjoint on $C_c^\infty(\mathcal M)$, integration by parts
gives
\[
\int_{\mathcal M} b_t\,\Delta\varphi\,d\mu_g
= \int_{\mathcal M} (\Delta b_t)\,\varphi\,d\mu_g
\;\ge\; c_0\int_{\mathcal M}\varphi\,d\mu_g .
\]
The family $\{b_t\}$ is $1$-Lipschitz and converges pointwise to $b$, so by
Lemma~\ref{lem:equi-lipschitz-conv} the convergence is uniform on
$\mathcal K$ and
\[
\int_{\mathcal M} b_t\,\Delta\varphi\,d\mu_g
\;\longrightarrow\;\int_{\mathcal M} b\,\Delta\varphi\,d\mu_g
= \int_{\mathcal M}(\Delta b)\,\varphi\,d\mu_g ,
\]
the last equality by integration by parts, legitimate since $b\in C^2$
(Proposition~\ref{prop:gradient-norm}). Hence $\Delta b\ge c_0$
distributionally; since $\Delta b$ is continuous, the inequality holds
pointwise. Indeed, were $\Delta b-c_0<0$ somewhere, continuity would give a
ball on which it stays negative, and a nonnegative test function supported
there would contradict the integral inequality.

For the upper bound, the Hessian comparison from below under
$K\ge-\kappa_1$ \cite[Thm.~6.4.3]{Petersen2016} gives, on $\mathcal K$ and
for $t$ large,
\[
\Delta r_t\le(n-1)\sqrt{\kappa_1}\,\coth\!\big(\sqrt{\kappa_1}\,r_t\big).
\]
Since $r_t\ge t-\max_{\mathcal K}d(\cdot,\gamma(0))\to\infty$ uniformly on
$\mathcal K$ and $\coth s\to1$, for every $\varepsilon>0$ and all $t$
large,
\[
\Delta r_t\le(n-1)\sqrt{\kappa_1}\,(1+\varepsilon)\qquad\text{on }\mathcal K .
\]
The same distributional argument yields
$\Delta b\le (n-1)\sqrt{\kappa_1}(1+\varepsilon)$ pointwise, and letting
$\varepsilon\downarrow0$ concludes the proof.
\end{proof}

\section{The main theorem: nonexistence of equal-level Busemann profiles}\label{sec:main}

\begin{definition}[Balanced double-well potential]\label{def:balanced}
A potential $W\in C^1(\R)$ is a \emph{double-well potential} with wells at
$\pm1$ if $\pm1$ are strict local minima of $W$,
\[
W'(\pm1)=0,
\]
and $W$ has no other local minimum in $[-1,1]$. It is \emph{balanced} if in
addition $W(-1)=W(1)$, and \emph{unbalanced} otherwise. When
$W\in C^2(\R)$, the wells are called \emph{non-degenerate} if
$W''(\pm1)>0$.
\end{definition}

The stronger condition is not used in Theorem~\ref{thm:main} or
Corollary~\ref{cor:curvature}, whose equal-level rigidity mechanism only
requires $W\in C^1(\R)$ and $W(-1)=W(1)$ for profiles joining the two
wells. Corollary~\ref{cor:limits-exist} requires in addition that the
critical points of $W$ be isolated; Proposition~\ref{prop:compatibility-front}
imposes the stronger unbalanced bistable hypotheses stated there; and the
normalization \eqref{eq:W-assumptions}, including non-degeneracy of the
wells, is imposed from Section~\ref{sec:warped} onward.

We assume throughout that the elliptic equation of interest is
$\Delta u = W'(u)$, where $W\in C^1(\mathbb R)$ is a prescribed potential;
its precise form will be specified in the application, as the reduction
below is independent of that choice.

\begin{definition}[Busemann profile]\label{def:profile}
A \emph{Busemann profile} is a function of the form $u=U\circ b$ with
$U\in C^2(\R)$.
\end{definition}

\begin{lemma}[Reduction]\label{lem:reduction}
A Busemann profile $u=U\circ b$ solves
\begin{equation}\label{eq:AC}
  \Delta u = W'(u)
\end{equation}
if and only if
\begin{equation}\label{eq:reduced}
  U''(b(x))+U'(b(x))\,\Delta b(x) = W'\big(U(b(x))\big)
  \qquad\text{for all }x\in\mathcal M .
\end{equation}
\end{lemma}
\begin{proof}
Immediate from Lemma~\ref{lem:invariant}.
\end{proof}

\begin{remark}[Consistency of the ansatz]\label{rem:consistency}
Fix $t_0\in\R$ and suppose $u=U\circ b$ solves \eqref{eq:AC}. Evaluating
\eqref{eq:reduced} at two points $x,y\in\Sigma_{t_0}=\{b=t_0\}$ and
subtracting, the terms $U''(t_0)$ and $W'(U(t_0))$ cancel, leaving
\[
U'(t_0)\bigl(\Delta b(x)-\Delta b(y)\bigr)=0 .
\]
Hence for each $t_0$, either $U'(t_0)=0$ or $\Delta b$ is constant on
$\Sigma_{t_0}$.

The ansatz is therefore self-limiting: wherever $U'(t)\neq0$, the mean
curvature of the corresponding horosphere $\Sigma_t$
(Corollary~\ref{cor:horosphere}) must be constant along that horosphere.
Note that the two-sided bound of Lemma~\ref{lem:comparison} does not
deliver this: it confines $\Delta b$ to an interval without making it
leaf-constant. Hyperbolic space satisfies the condition in the strongest
form, $\Delta b\equiv n-1$ (Example~\ref{ex:H2}). Theorem~\ref{thm:main}
requires no such leaf-constancy hypothesis: it assumes only
$\inf_{\mathcal M}\Delta b>0$, and rules out nonconstant profiles between
equal potential levels regardless of the leaf structure. Under the
equal-potential hypotheses of Theorem~\ref{thm:main}, nonconstant profiles
are nevertheless excluded without assuming this leaf-constancy a priori: as
already noted in the introduction, the profile ansatz $u=U\circ b$ itself
is retained, and the result extends the \emph{invariant-profile}
nonexistence mechanism of \cite{BirindelliMazzeo2009}, not their symmetry
theorems, which force certain general solutions with invariant asymptotic
data to be invariant.
\end{remark}

\begin{theorem}[Rigidity of Equal-Potential Busemann Profiles]\label{thm:main}
Let $(\mathcal{M}^n, g)$, with $n \ge 2$, be a smooth Cartan--Hadamard
manifold, and let $b=b_\xi$ be a Busemann function on $\mathcal M$
exhibiting a uniformly strictly positive Laplacian, namely
\begin{equation}
    c := \inf_{\mathcal{M}} \Delta_g b > 0.
\end{equation}
Let $W \in C^1(\mathbb{R})$ be a prescribed potential. If $u = U \circ b$,
where $U \in C^2(\mathbb{R})$, is a Busemann profile solving the
semilinear elliptic equation $\Delta_g u = W'(u)$ on $\mathcal{M}$, and the
asymptotic limits $\ell_\pm := \lim_{t \to \pm\infty} U(t)$ exist and are
finite, then the profile satisfies the one-sided energy constraint
 \begin{equation}\label{eq:climbing}
   W(\ell_-)\;\le\;W(\ell_+),
 \end{equation}
where $\ell_-$ corresponds to the limit as $t\to-\infty$ (towards the ideal
boundary point $\xi$), and $\ell_+$ to $t\to+\infty$ (away from $\xi$).
Furthermore, equality $W(\ell_-) = W(\ell_+)$, e.g.\ if
$\ell_\pm\in\{-1,+1\}$ are the two wells of a balanced double-well
potential (Definition~\ref{def:balanced}), holds if and only if $U$ is
identically constant. Consequently, for any balanced double-well
potential, no nonconstant Busemann profile connecting the energy minima
can exist on $\mathcal{M}$.
\end{theorem}
\begin{proof}
Let $\varphi:\R\to\M$ be an integral curve of $\nabla b$, parametrized so
that $b(\varphi(0))=0$; by Proposition~\ref{prop:global}(ii), $\varphi$ is
a complete unit-speed geodesic with $b(\varphi(t))=t$. Evaluating
\eqref{eq:reduced} along $\varphi$ and writing
$\alpha(t):=\dg b(\varphi(t))\ge c>0$, continuous since $b\in C^2$, we find
that $U$ satisfies
\begin{equation}\label{eq:ode-flowline}
  U''(t)+\alpha(t)\,U'(t)=W'(U(t)),\qquad t\in\R .
\end{equation}
If $U$ is nonconstant, Remark~\ref{rem:consistency} shows that $\dg b$ is
constant on each horosphere met by $\{U'\neq0\}$, so $\alpha$ does not
depend on the chosen flow line there; in any case the argument below uses
only the bound $\alpha\ge c$.

Define
\begin{equation}\label{eq:hamiltonian}
  H(t):=\tfrac12\,U'(t)^2-W(U(t))\in C^1(\R),
\end{equation}
so that by \eqref{eq:ode-flowline},
\begin{equation}\label{eq:dissipation}
  H'(t)=U'\big(U''-W'(U)\big)=-\alpha(t)\,U'(t)^2\;\le\;-c\,U'(t)^2\;\le\;0 .
\end{equation}
Thus $H$ is nonincreasing and its limits $H(\pm\infty)$ exist in the
extended reals.

We identify these limits by a sequential argument. Since $U$ has finite
limits $\ell_\pm$, the mean value theorem on $[m,m+1]$ provides
$\tau_m\in(m,m+1)$ with
\[
  U'(\tau_m)=U(m+1)-U(m)\longrightarrow0 \qquad(m\to\pm\infty).
\]
As $\tau_m\to\pm\infty$ we have $U(\tau_m)\to\ell_\pm$, so by continuity of
$W$,
\[
  H(\tau_m)=\tfrac12U'(\tau_m)^2-W(U(\tau_m))\longrightarrow -W(\ell_\pm).
\]
Being monotone, $H$ has limits at $\pm\infty$ that agree with those along
any sequence tending to $\pm\infty$; hence
\[
  H(-\infty)=-W(\ell_-),\qquad H(+\infty)=-W(\ell_+),
\]
both finite.

To obtain the integral identity, for any $R>0$,
\[
  H(R)-H(-R)=\int_{-R}^{R} H'(t)\,dt=-\int_{-R}^{R}\alpha(t)\,U'(t)^2\,dt .
\]
The integrand $\alpha U'^2$ is nonnegative, so by monotone convergence the
improper integral exists in $[0,+\infty]$; its value is finite because
$H(\pm\infty)$ are. Letting $R\to\infty$,
\begin{equation}\label{eq:integral-identity}
  W(\ell_+)-W(\ell_-)=\int_{-\infty}^{+\infty} \alpha(t)\,U'(t)^2\,dt \;\ge\;0,
\end{equation}
which is \eqref{eq:climbing}. Equality holds if and only if the integral
vanishes; since $\alpha\ge c>0$, this forces $U'\equiv0$, i.e.\ $U$
constant, say $U\equiv q$; evaluating \eqref{eq:ode-flowline} at any $t$
then gives $W'(q)=0$. Conversely, if $U\equiv q$ is any constant solution,
then $\ell_-=\ell_+=q$ and equality holds trivially. In particular, every
nonconstant profile satisfies $W(\ell_-)<W(\ell_+)$.
\end{proof}

\begin{corollary}\label{cor:curvature}
Let $\mathcal M$ be Cartan--Hadamard with sectional curvature
$K\le-\kappa_2<0$, and let $W\in C^1$ be a balanced double-well potential
(Definition~\ref{def:balanced}). Then no nonconstant Busemann profile
connects the wells.
\end{corollary}
\begin{proof}
Lemma~\ref{lem:comparison} gives
$\inf_{\mathcal M}\Delta b\ge(n-1)\sqrt{\kappa_2}>0$; apply
Theorem~\ref{thm:main}.
\end{proof}

Since $\Delta_g b\ge0$ holds automatically
(Proposition~\ref{prop:flow-geodesic}), it is natural to ask how much of
Theorem~\ref{thm:main} survives when the strict lower bound $c>0$ is
dropped. At the $C^1$ level of Theorem~\ref{thm:main} we do not know; but
assuming $W'$ locally Lipschitz suffices to replace that bound by the far
weaker $\Delta_g b\not\equiv0$.

\begin{corollary}[Nontrivial Busemann drift]\label{cor:nonnegative-drift}
Let $(\mathcal M,g)$ be Cartan--Hadamard and let $b=b_\xi$ be a Busemann
function with
\[
\Delta_g b\not\equiv0 .
\]
Let $W\in C^1(\mathbb R)$ with $W'$ locally Lipschitz (in particular, any
$W\in C^2(\mathbb R)$). Let $u=U\circ b$, with $U\in C^2(\mathbb R)$,
solve $\Delta_g u=W'(u)$, and suppose the limits
$\ell_\pm=\lim_{t\to\pm\infty}U(t)$ exist, are finite, and satisfy
$W(\ell_-)=W(\ell_+)$. Then $U$ is constant.
\end{corollary}
\begin{proof}
Choose $x_0\in\mathcal M$ with $\Delta_g b(x_0)>0$, set $t_0:=b(x_0)$, and
let $\varphi$ be the integral curve of $\nabla b$ with $\varphi(t_0)=x_0$;
by Proposition~\ref{prop:global}(ii) it is complete with
$b(\varphi(t))=t$. Writing $\alpha(t):=\Delta_g b(\varphi(t))$, continuous
since $b\in C^2(\mathcal M)$ (Proposition~\ref{prop:gradient-norm}) and
nonnegative by Proposition~\ref{prop:flow-geodesic}, equation
\eqref{eq:ode-flowline} holds along $\varphi$, and $H'=-\alpha (U')^2\le0$
by the same computation as in \eqref{eq:dissipation}, now using only
$\alpha\ge0$.

The sequential argument of Theorem~\ref{thm:main}, which uses only the
existence of the finite limits $\ell_\pm$ and the monotonicity of $H$, not
the bound $\alpha\ge c$, yields $H(\mp\infty)=-W(\ell_\mp)$. The
equal-level hypothesis forces $H(-\infty)=H(+\infty)$, so $H$ is constant
and
\[
\alpha(t)\,U'(t)^2\equiv0 \qquad (t\in\mathbb R).
\]
Since $\alpha(t_0)>0$ and $\alpha$ is continuous, $\alpha>0$ on an open
interval $I\ni t_0$; hence $U'\equiv0$ on $I$, so $U\equiv q$ there for
some constant $q$, and evaluating \eqref{eq:ode-flowline} on $I$ gives
$W'(q)=0$.

Thus $(U,U')(t_0)=(q,0)$, and the constant function $q$ solves
\eqref{eq:ode-flowline} with that same datum. The field
$(U,V)\mapsto(V,\,W'(U)-\alpha(t)V)$ is continuous in $t$ and locally
Lipschitz in $(U,V)$, so the Cauchy problem has a unique solution; hence
$U\equiv q$ on all of $\mathbb R$.
\end{proof}

\begin{remark}[The exceptional case is exactly the flat splitting]
\label{rem:dichotomy}
Throughout this remark $W$ is as in Corollary~\ref{cor:nonnegative-drift},
with $W'$ locally Lipschitz. The hypothesis $\Delta_g b\not\equiv0$ is
sharp, and its failure is completely understood. Indeed $\Hess b\ge0$
(Proposition~\ref{prop:flow-geodesic}), so $\Delta_g b\equiv0$ forces
$\Hess b\equiv0$: a positive semidefinite form of vanishing trace
vanishes. By Remark~\ref{rem:orientation}, $\mathcal M$ then splits
isometrically as $\mathbb R\times\mathcal N$ with $b$ equal, up to an
additive constant, to minus the projection onto the first factor, so
$\Delta_g(U\circ b)=U''\circ b$ and every solution of $U''=W'(U)$
produces a Busemann profile. For a balanced double-well potential
(Definition~\ref{def:balanced}) one has $W>W(\pm1)$ on $(-1,1)$,
otherwise the minimum of $W$ over $[-1,1]$ would be attained at an
interior point, which would be a further local minimum there, and the
classical heteroclinic joining $-1$ to $1$ is such a profile: nonconstant,
with equal potential levels.

Hence, for $W'$ locally Lipschitz, Corollary~\ref{cor:nonnegative-drift}
is a dichotomy: an equal-level Busemann profile is constant unless $b$ is
affine, in which case $\mathcal M$ splits and rigidity genuinely fails, as
the balanced double well above shows.
\end{remark}

The existence of the finite limits $\ell_\pm$ is an explicit hypothesis of
Theorem~\ref{thm:main}; it is not automatic under the general assumption
$W\in C^1(\mathbb R)$, where no sign or growth condition is imposed
outside a bounded range. For a \emph{bounded} profile $U$, the dissipation
identity together with $\alpha\ge c>0$, combined with a bound on the
Hamiltonian $H$, yields
\[
U'\in L^2(\mathbb R).
\]
Under two-sided negative curvature pinching, $-\kappa_1\le K \le-\kappa_2<0$,
and the additional assumption that $W$ has isolated critical points,
Corollary~\ref{cor:limits-exist} upgrades this to the existence of finite
limits $\ell_\pm$ with $W'(\ell_\pm)=0$ for every bounded Busemann
profile. We do not pursue conditions guaranteeing boundedness of $U$
itself, which would require growth hypotheses on $W'$ outside the range of
interest.

We first record an elementary fact about limit sets.

\begin{lemma}[Connectedness of the limit set]\label{lem:omega-connected}
Let $f:[0,\infty)\to\R$ be continuous and bounded. Then its set of
accumulation points as $t\to+\infty$,
\[
  \omega(f):=\{p\in\R:\ f(t_n)\to p\ \text{for some}\ t_n\to+\infty\},
\]
is a nonempty closed interval. The same holds as $t\to-\infty$.
\end{lemma}
\begin{proof}
Nonemptiness follows from the Bolzano--Weierstrass theorem. To see
closedness, take a sequence $p_k\in\omega(f)$ with $p_k\to p$. For each
$k$, choose $t_{k,j}\to\infty$ with $f(t_{k,j})\to p_k$. A diagonal
argument yields $t_k\to\infty$ with $f(t_k)\to p$, hence $p\in\omega(f)$.

For connectedness, let $p<q$ lie in $\omega(f)$ and let $p<r<q$. Choose
sequences $s_n\to\infty$ and $u_n\to\infty$ such that $f(s_n)\to p$ and
$f(u_n)\to q$. Passing to subsequences if necessary, we may assume that
$s_n<u_n<s_{n+1}$ for all $n$ large (interlacing). For sufficiently large
$n$, we have $f(s_n)<r<f(u_n)$. By the intermediate value theorem, there
exists $v_n\in(s_n,u_n)$ with $f(v_n)=r$. Since $v_n\to\infty$, it follows
that $r\in\omega(f)$. Therefore $\omega(f)$ is connected, hence an
interval (closed by the first part).
\end{proof}

\begin{corollary}[Existence of limits under two-sided pinching]
\label{cor:limits-exist}
Let $\mathcal M$ be Cartan--Hadamard with $-\kappa_1 \le K \le -\kappa_2 
0$, and let $W \in C^1(\mathbb R)$ have isolated critical points. Then
every bounded Busemann profile $u = U \circ b$ solving \eqref{eq:AC} has
finite limits $\ell_\pm = \lim_{t \to \pm\infty} U(t)$, with
$W'(\ell_\pm) = 0$. In particular, the hypothesis of finite limits in
Theorem~\ref{thm:main} may be replaced by boundedness of $U$ under these
assumptions.
\end{corollary}
\begin{proof}
Fix a flow line $\varphi$ as in the proof of Theorem~\ref{thm:main}, so
that $U$ solves \eqref{eq:ode-flowline} with
$\alpha(t) = \Delta b(\varphi(t))$. By Lemma~\ref{lem:comparison},
\[
c := (n-1)\sqrt{\kappa_2} \ \le\ \alpha(t)\ \le\ (n-1)\sqrt{\kappa_1} =: C
\qquad (t \in \mathbb R),
\]
so $\alpha$ is bounded above and below by positive constants. Let
$I \subset \mathbb R$ be the closure of the range of $U$, a compact
interval.

First, we establish boundedness of $H$, $U'$, $U''$, and the
integrability of $U'^2$. From \eqref{eq:dissipation}, $H$ is
nonincreasing, since $\alpha\ge c>0$. Writing $B:=\max_{x\in I}|x|$, so
that $|U|\le B$ on $\mathbb R$, we bound $H$ on both sides.

\emph{Lower bound.} Since $W$ is continuous on the compact set $I$,
\[
H(t) = \tfrac12 U'(t)^2 - W(U(t)) \;\ge\; -\max_I W \qquad (t\in\mathbb R).
\]

\emph{Upper bound.} Fix $t\in\mathbb R$. By the mean value theorem on
$[t-1,t]$ there is $s\in(t-1,t)$ with
\[
U'(s)=U(t)-U(t-1),\qquad\text{hence}\qquad |U'(s)|\le 2B .
\]
Since $H$ is nonincreasing and $s<t$,
\[
  H(t)\le H(s)
  =\tfrac12U'(s)^2-W(U(s))
  \le2B^2+\max_I|W|
  =:H_+ .
\]
Hence $H$ is bounded on $\mathbb R$; being monotone, it has finite limits
$H(\pm\infty)$, with
\[
-\max_I W\le H(+\infty)\le H(-\infty)\le H_+ .
\]
Consequently,
\[
U'(t)^2 = 2\bigl(H(t) + W(U(t))\bigr)
\]
is bounded, hence so is $U'$. Moreover, $U'' = W'(U) - \alpha U'$ is
bounded because $W'$ is continuous on $I$ and $\alpha, U'$ are bounded.
Furthermore, \eqref{eq:dissipation} gives $c\, U'^2 \le -H'$, so for
every $R > 0$
\[
c \int_{-R}^{R} U'(t)^2 \, dt \le H(-R) - H(R)
\le H(-\infty) - H(+\infty),
\]
and letting $R \to \infty$,
\[
\int_{\mathbb R} U'(t)^2 \, dt
\le \frac{H(-\infty) - H(+\infty)}{c} < \infty.
\]

Next, we show that $U'(t) \to 0$ as $t \to \pm\infty$. Let
$L:=\|U''\|_{L^\infty(\mathbb R)}$. If $L=0$, then $U'$ is constant on
$\mathbb R$; since $U'\in L^2(\mathbb R)$, necessarily $U'\equiv0$, and
the claim is trivial. Assume henceforth $L>0$; then $U'$ is
$L$-Lipschitz. Suppose, for contradiction, that $U'(t) \not\to 0$ as
$t \to +\infty$. Then there exist $\varepsilon > 0$ and $t_n \to +\infty$
with $|U'(t_n)| \ge \varepsilon$. Put $\delta := \varepsilon/(2L)$. For
$|t - t_n| \le \delta$,
\[
|U'(t)| \ge |U'(t_n)| - L|t - t_n| \ge \frac{\varepsilon}{2}.
\]
After passing to a subsequence with $t_{n+1} > t_n + 2\delta$, the
intervals $[t_n - \delta, t_n + \delta]$ are pairwise disjoint, and each
contributes at least
\[
2\delta \cdot \left(\frac{\varepsilon}{2}\right)^2
= \frac{\varepsilon^3}{4L}
\]
to $\int_{\mathbb R} U'^2$, contradicting its finiteness. The same
argument applies as $t \to -\infty$.

Now we prove that every accumulation point of $U$ as $t \to +\infty$ is a
critical point of $W$. Integrating \eqref{eq:ode-flowline} over
$[t, t+1]$,
\[
\int_t^{t+1} W'(U(s)) \, ds
= U'(t+1) - U'(t) + \int_t^{t+1} \alpha(s) U'(s) \, ds.
\]
As $t \to +\infty$, the right-hand side tends to $0$: the first two
terms by the decay of $U'$, and the integral because
$|\alpha U'| \le C \sup_{[t,t+1]} |U'| \to 0$. Now let $p \in \omega(U)$,
say $U(t_n) \to p$ with $t_n \to +\infty$. Put
$\varepsilon_n:=\sup_{[t_n,t_n+1]}|U'|$; by the decay of $U'$
established above, $\varepsilon_n\to0$, so
\[
  \sup_{s\in[t_n,t_n+1]}|U(s)-p|
  \;\le\;|U(t_n)-p|+\varepsilon_n\;\longrightarrow\;0 ,
\]
i.e.\ $U\to p$ uniformly on $[t_n,t_n+1]$. As $W'$ is uniformly
continuous on $I$, it follows that $W'(U) \to W'(p)$ uniformly there,
whence
\[
0 = \lim_{n \to \infty} \int_{t_n}^{t_n + 1} W'(U(s)) \, ds = W'(p).
\]

Finally, by Lemma~\ref{lem:omega-connected}, the set $\omega(U)$ (for
$t \to +\infty$) is a closed interval; by the previous paragraph it is
contained in the critical set of $W$, which is discrete. An interval
contained in a discrete set is a singleton, so $\omega(U) = \{\ell_+\}$
and $U(t) \to \ell_+$ with $W'(\ell_+) = 0$. The same reasoning applied
to $t \to -\infty$ yields $\ell_-$.
\end{proof}

\begin{remark}[Sharpness: unequal levels admit nonconstant profiles]
\label{rem:sharpness}
The conclusion of Theorem~\ref{thm:main} is sharp: nonconstant profiles
can exist when $W(\ell_-)<W(\ell_+)$, and this is compatible with a
balanced double well, the constraint falls on the levels of the limits,
not on the shape of $W$.

Take $W\in C^2(\R)$ with wells at $\pm1$ in the sense of
Definition~\ref{def:balanced}, an additional critical point at $0$ with
$W(0)>0$, and $W(-1)=W(1)<W(0)$. Note that this $W$ is \emph{balanced} in
the sense of Definition~\ref{def:balanced}: the point is precisely that
Theorem~\ref{thm:main} constrains the levels $W(\ell_\pm)$ of the limits,
not the shape of $W$, and here the profile joins a well to the
intermediate critical value $0$, which sits strictly higher. (Precisely:
the potential $W$ is taken to satisfy the hypotheses of
\cite[eq.~(1.2), Prop.~2.2]{BirindelliMazzeo2009}, not merely the weaker
properties stated above.) On $\Hy^n$ in the upper half-space model of
Example~\ref{ex:H2}, with height coordinate $y$, $b=-\log y$ and
$\xi=\infty$, Busemann profiles $u=U\circ b$ are exactly the solutions
invariant under the parabolic subgroup fixing $\xi$. Birindelli--Mazzeo
\cite[Prop.~2.2, second assertion]{BirindelliMazzeo2009} construct such a
solution, unique up to the dilations $y\mapsto\lambda y$, $\lambda>0$
(which act on $U$ as the translations $b\mapsto b-\log\lambda$), with
\[
  \lim_{y\to+\infty} u = 1, \qquad \lim_{y\to 0} u = 0 .
\]
Since $b\to-\infty$ as $y\to+\infty$, this reads $\ell_-=1$ and
$\ell_+=0$ in Busemann coordinates, so $W(\ell_-)<W(\ell_+)$: the profile
climbs the potential in the direction of increasing $b$, as
Theorem~\ref{thm:main} requires. Note that it joins a well to the local
maximum, not two wells, precisely the configuration the theorem permits.
The connection between the two wells is excluded, and that is the first
assertion of \cite[Prop.~2.2]{BirindelliMazzeo2009}, recovered here by
Corollary~\ref{cor:curvature}.

The profile need not be monotone. In the logarithmic height coordinate
$\xi=\log y=-b$ used by \cite{BirindelliMazzeo2009}, the linearization at
the origin has characteristic equation
$\lambda^2-(n-1)\lambda-W''(0)=0$; when $W''(0)<-(n-1)^2/4$, the origin
is an unstable spiral in the $\xi$-orientation. Equivalently, in our
Busemann coordinate $b=-\xi$, it is a \emph{stable} spiral, approached as
$b\to+\infty$, where indeed $U(b)\to0$. Birindelli--Mazzeo further
identify a critical value $\gamma$ such that the profile remains
strictly positive if $-W''(0)\le\gamma$ and changes sign if
$-W''(0)>\gamma$; thus neither positivity nor monotonicity is automatic
once $|W''(0)|$ is large. This is immaterial here, since
Theorem~\ref{thm:main} requires only the existence of the limits
$\ell_\pm$.
\end{remark}

When $\Delta b\equiv c$ is constant (as on $\Hy^n$, where $c=n-1$),
equation \eqref{eq:ode-flowline} is the traveling-wave profile equation
$U''+cU'=W'(U)$ of the parabolic Allen--Cahn flow with wave speed $c$.
For a balanced double well, the only speed admitting a two-well front is
$c=0$: integrating the dissipation identity $H'=-c\,(U')^2$ over $\R$ and
using $H(\pm\infty)=-W(\ell_\pm)=-W(\pm1)$ gives
$0=H(+\infty)-H(-\infty)=-c\int_\R (U')^2$, forcing $c=0$ for
nonconstant $U$. Theorem~\ref{thm:main} is the geometric incarnation of
this classical fact (see Fife--McLeod~\cite{FifeMcLeod1977} for the
parabolic front theory), with the geometry supplying a variable,
uniformly positive speed.

\subsection*{Exact horospherical fronts at the compatibility point}\label{subsec:exact-fronts}

\begin{proposition}[Exact monotone horospherical fronts]
\label{prop:compatibility-front}
Assume that $\Delta_g b\equiv c$ for a constant $c>0$. Let
$W\in C^2(\mathbb R)$ be a double-well potential whose critical points are
three simple zeros $-1<m<1$ of $W'$, with
\[
W''(\pm1)>0,\qquad W''(m)<0,\qquad W(-1)<W(1),
\]
and let $\tau^*(W')>0$ denote the unique Fife--McLeod speed for which
\begin{equation}\label{eq:fife-mcleod-front}
U''+\tau U'=W'(U)
\end{equation}
admits a strictly increasing solution $U\in C^2(\mathbb R)$ satisfying
$U(-\infty)=-1$ and $U(+\infty)=1$. Then the following are equivalent:
\begin{enumerate}
\item[\textup{(i)}] $c=\tau^*(W')$;
\item[\textup{(ii)}] there exists a strictly increasing
$U\in C^2(\mathbb R)$, with $U(-\infty)=-1$, $U(+\infty)=1$, such that
$u:=U\circ b$ solves $\Delta_g u=W'(u)$ on $\mathcal M$.
\end{enumerate}
Whenever these equivalent conditions hold, every level set $\{u=s\}$,
$s\in(-1,1)$, is a horosphere $\Sigma_{t_s}=\{b=t_s\}$, with constant
mean curvature $H_{\Sigma_{t_s}}=c$ with respect to $N=\nabla b$; for
$n=2$, its signed geodesic curvature is $k_g=-c$.
\end{proposition}
\begin{proof}
(i)$\Rightarrow$(ii): By the Fife--McLeod theorem
\cite[\S2, Cor.~2.3 and Thm.~2.4]{FifeMcLeod1977}, the unbalanced
condition $W(-1)<W(1)$ singles out a unique speed $\tau^*(W')>0$ for
which \eqref{eq:fife-mcleod-front} has a strictly increasing solution
with the prescribed limits; such a $U$ is a bijection of $\mathbb R$
onto $(-1,1)$. The compatibility hypothesis $c=\tau^*(W')$ then turns
\eqref{eq:fife-mcleod-front} into $U''+c\,U'=W'(U)$. Since
$b\in C^2(\mathcal M)$ and $U\in C^2(\mathbb R)$,
$u:=U\circ b\in C^2(\mathcal M)$, and Lemma~\ref{lem:invariant} gives
\[
\Delta_g u = U''(b) + U'(b)\,\Delta_g b = U''(b) + c\,U'(b) = W'(U(b)) = W'(u).
\]

(ii)$\Rightarrow$(i): Choose a complete integral curve $\varphi$ of
$\nabla b$, parametrized so that $b(\varphi(t))=t$. Evaluating the
profile equation along $\varphi$, Lemma~\ref{lem:invariant} and
$\Delta_g b\equiv c$ give
\[
U''(t)+c\,U'(t)=W'(U(t)) \qquad (t\in\mathbb R).
\]
Since $U$ is strictly increasing with $U(\mp\infty)=\mp1$, uniqueness of
the Fife--McLeod speed \cite[\S2, Cor.~2.3 and Thm.~2.4]{FifeMcLeod1977}
yields $c=\tau^*(W')$.

For either direction, given $s\in(-1,1)$, let $t_s$ be the unique real
number with $U(t_s)=s$; since $U$ is a bijection onto $(-1,1)$, such
$t_s$ exists and is unique. Then $\{u=s\}=\{U(b)=s\}=\{b=t_s\}=
\Sigma_{t_s}$, a horosphere. By Corollary~\ref{cor:horosphere},
$H_{\Sigma_t}=\Delta_g b=c$ for every $t$, and $k_g=-c$ when $n=2$.
\end{proof}

As a by-product, under the standing hypothesis $\Delta_g b\equiv c$ of
this proposition, elliptic regularity applied to $\Delta_g b=c$ shows
that $b$ is in fact smooth, resolving any regularity question in this
regime independently of Proposition~\ref{prop:gradient-norm}.

\subsection*{Extension: heterogeneous drift via the Busemann dictionary}
\label{subsec:heterogeneous}

Proposition~\ref{prop:compatibility-front} imposes a stringent
constraint: horospherical fronts exist exactly when $c=\tau^*(W')$, with
$c=\Delta_g b_\xi$ constant. Pinching alone does not deliver constancy:
by Lemma~\ref{lem:comparison} it only confines $\Delta_g b_\xi$ to the
interval $[(n-1)\sqrt{\kappa_2},(n-1)\sqrt{\kappa_1}]$. The natural
intermediate hypothesis is that $\Delta_g b_\xi$ be constant on each
horosphere, i.e.\ a function of $b_\xi$ alone. This is precisely the
condition that makes the reduced coefficient independent of the chosen
flow line and hence turns the profile equation into a single
nonautonomous ODE. By Remark~\ref{rem:consistency}, such leaf-constancy
is necessarily present, for any Busemann profile solving the equation,
at every level where $U'(t)\neq0$.

\begin{proposition}[Geometric dictionary: pinching as heterogeneous
media]\label{prop:dictionary}
Let $\mathcal M$ satisfy $-\kappa_1\le K\le-\kappa_2<0$ and suppose
\begin{equation}\label{eq:busemann-alpha}
  \Delta_g b_\xi(x) = \alpha\bigl(b_\xi(x)\bigr)
\end{equation}
for some continuous $\alpha$. Then necessarily
$\alpha(\R)\subset[(n-1)\sqrt{\kappa_2},\,(n-1)\sqrt{\kappa_1}]$ by
Lemma~\ref{lem:comparison}, and:
\begin{enumerate}
\item[\textup{(i)}] every solution $U\in C^2(\R)$ of
\begin{equation}\label{eq:nonauto-main}
  U''(t)+\alpha(t)\,U'(t)=W'(U(t))
\end{equation}
yields, via $u:=U\circ b_\xi$, a Busemann profile solving
$\Delta_g u=W'(u)$ on $\mathcal M$;
\item[\textup{(ii)}] conversely, if $u=U\circ b_\xi$ solves
$\Delta_g u=W'(u)$ on $\mathcal M$, then $U$ solves
\eqref{eq:nonauto-main} on all of $\R$. In either case,
$u(\mathcal M)=U(\R)$, so $u$ is bounded if and only if $U$ is;
\item[\textup{(iii)}] if $U$ is strictly monotone with finite limits
$\ell_\pm$, then for every
$s\in(\min\{\ell_-,\ell_+\},\max\{\ell_-,\ell_+\})$ the level set
$\{u=s\}$ is the single horosphere $\{b_\xi=U^{-1}(s)\}$.
\end{enumerate}
Consequently, existence and nonexistence for \eqref{eq:nonauto-main} are
equivalent to existence and nonexistence \emph{within the class of
Busemann profiles} on $\mathcal M$, and the correspondence preserves
boundedness, finite asymptotic limits, and monotonicity.
\end{proposition}
\begin{proof}
By Lemma~\ref{lem:invariant} and \eqref{eq:busemann-alpha},
\[
\Delta_g(U\circ b_\xi) = U''(b_\xi)+\alpha(b_\xi)\,U'(b_\xi).
\]
This proves (i). Conversely, if $u=U\circ b_\xi$ solves the PDE, the
same identity gives \eqref{eq:nonauto-main} for every
$t\in b_\xi(\mathcal M)$; by Proposition~\ref{prop:global}(ii),
$b_\xi(\mathcal M)=\mathbb R$, so the ODE holds on all of $\mathbb R$,
proving (ii). Since $u(\mathcal M)=U(b_\xi(\mathcal M))=U(\mathbb R)$,
boundedness of one is equivalent to that of the other. For (iii), a
strictly monotone continuous function with finite limits $\ell_\pm$ is a
bijection of $\mathbb R$ onto
$\bigl(\min\{\ell_-,\ell_+\},\max\{\ell_-,\ell_+\}\bigr)$, so for each
$s$ in that interval there is a unique $t_s=U^{-1}(s)$, and
$\{u=s\}=\{U(b_\xi)=s\}=\{b_\xi=t_s\}$.
\end{proof}

\begin{proposition}[Warped realization under a Riccati bound]
\label{prop:concrete-realization}
Let $\beta:\mathbb R\to(0,\infty)$ be smooth and non-decreasing, and
suppose in addition that
\[
0<\beta_-:=\lim_{t\to-\infty}\beta(t)\;\le\;\beta(t)\;\le\;
\beta_+:=\lim_{t\to+\infty}\beta(t)<\infty
\qquad(t\in\mathbb R),
\]
both limits existing by monotonicity and being finite by this
hypothesis. Set
\[
  \mathcal M = \mathbb R_t \times \mathbb R^{n-1}_s,
  \qquad g = dt^2 + f(t)^2 |ds|^2,
  \qquad f(t) = \exp\!\Big(\int_0^t \beta\Big).
\]
If $\beta$ satisfies the Riccati inequality
\begin{equation}\label{eq:riccati}
  \beta'(t) \;\le\; \beta_+^2 - \beta(t)^2 \qquad \text{for all } t,
\end{equation}
then $(\mathcal M,g)$ is a complete Cartan--Hadamard manifold, pinched as
$-\beta_+^2\le K\le-\beta_-^2$; the coordinate $t$ is exactly the
Busemann function $b_\xi$ at the ideal point $\xi$ represented by the
rays $\tau\mapsto(-\tau,s)$ ($\tau\ge0$, any $s\in\mathbb R^{n-1}$),
normalized by $b_\xi(0,s_0)=0$; and $\Delta_g b_\xi=(n-1)\beta(t)$,
giving \eqref{eq:busemann-alpha} with $\alpha=(n-1)\beta$. In
particular, nonconstant choices of $\beta$ satisfying the above
assumptions provide Cartan--Hadamard manifolds with nonconstant Busemann
drift.
\end{proposition}
\begin{proof}
\emph{Curvature.} The sectional curvatures are
$K_{\mathrm{mix}} = -f''/f = -(\beta' + \beta^2)$ on planes
$\partial_t \wedge \partial_{s_i}$, and
$K_{\mathrm{tan}} = -(f'/f)^2 = -\beta^2$ on tangential planes (present
when $n \ge 3$). Since $\beta \ge \beta_-$, the tangential curvatures
lie in $[-\beta_+^2, -\beta_-^2]$ automatically. The mixed curvatures
satisfy $K_{\mathrm{mix}} \le -\beta_-^2$ (as $\beta' \ge 0$), but the
lower bound $K_{\mathrm{mix}} \ge -\beta_+^2$ holds if and only if
$\beta$ obeys \eqref{eq:riccati}, which does \emph{not} follow from
monotonicity alone (without the stated two-sided bound, monotonicity
and positivity alone only guarantee $\beta_-\in[0,\infty)$ and
$\beta_+\in(0,\infty]$, e.g.\ $\beta(t)=e^t/(1+e^t)$ is smooth,
positive, and non-decreasing, yet $\beta_-=0$; the identification of
$t$ with a Busemann function below uses $\beta_->0$ directly, and
\eqref{eq:riccati} itself is vacuous unless $\beta_+<\infty$). Under
\eqref{eq:riccati} the metric is therefore pinched as
$-\beta_+^2 \le K \le -\beta_-^2$.

\emph{Completeness.} Since $g\ge dt^2$ pointwise, every curve $\eta$
satisfies $\mathrm{length}_g(\eta)\ge\bigl|\!\int t'\bigr|\ge|\Delta t|$
along it, so $d_g(x,y)\ge|t(x)-t(y)|$ for all $x,y\in\mathcal M$; as in
the warped-line class (A) of Section~\ref{sec:warped}, for a curve
$\gamma(u)=(t(u),s(u))$ of $g$-length $L$ starting at $t(0)=t_0$,
$|t(u)-t_0|\le L$ for every $u$, so the entire curve stays in the slab
$[t_0-L,t_0+L]$. On that slab the continuous positive function $f$ has
a positive minimum $m$, and $g\ge dt^2+m^2|ds|^2$ there controls the
$s$-coordinate as well. Hence a $d_g$-Cauchy sequence is Cauchy in the
$t$- and, once $t$ is bounded, in the $s$-coordinates, and thus
converges in $(\mathcal M,g)$; by Hopf--Rinow, $(\mathcal M,g)$ is
complete, hence Cartan--Hadamard (the curvature computation above gives
$K\le0$).

\emph{Identification of $t$ with a Busemann function.} Fix
$s_0\in\mathbb R^{n-1}$ and let $\gamma(\tau):=(-\tau,s_0)$,
$\tau\ge0$, the unit-speed geodesic ray obtained by restricting the
flow line of $\partial_t$ through $(0,s_0)$ to $\tau\ge0$ with reversed
orientation. Since $\beta\ge\beta_->0$,
\[
f(-R)=\exp\!\left(-\int_{-R}^{0}\beta(\tau)\,d\tau\right)
\le e^{-\beta_- R}\longrightarrow0 \qquad (R\to\infty),
\]
so any two vertical rays $\{s=s_0\}$ and $\{s=s_1\}$ are asymptotic:
their $s$-separation at parameter $-R$ has length $f(-R)|s_0-s_1|\to0$,
and hence every ray $\tau\mapsto(-\tau,s)$, for any $s$, represents the
same ideal point $\xi$ as $\gamma$. For $x=(t,s)\in\mathcal M$ and
$R>0$ large, the vertical-then-horizontal path from $x$ through
$(-R,s)$ to $\gamma(R)=(-R,s_0)$ has length $t+R+f(-R)|s-s_0|$, giving
the upper bound
\[
d(x,\gamma(R))-R \;\le\; t+f(-R)|s-s_0| \;\longrightarrow\; t
\qquad(R\to\infty),
\]
while $g\ge dt^2$ forces every curve joining $x$ to $\gamma(R)$ to
traverse a $t$-displacement of at least $t+R$, so
$d(x,\gamma(R))\ge t+R$ and hence $d(x,\gamma(R))-R\ge t$ for all $R$.
Combining the two bounds,
\[
d(x,\gamma(R))-R \;\longrightarrow\; t \qquad (R\to\infty),
\]
so $t$ is exactly the Busemann function $b_\xi$ (with the
normalization $b_\xi(0,s_0)=0$) associated to this ray. Since
$|\nabla t|\equiv1$, Proposition~\ref{prop:flow-geodesic} gives
$\Hess t(\partial_t,\cdot)=0$, so $\Delta_g t$ is the trace of the
second fundamental form of the leaf. For the metric
$g=dt^2+f(t)^2|ds|^2$ one has $\Gamma^{t}_{ij}=-f f'\delta_{ij}$,
hence, in the leaf orthonormal frame $e_i:=f^{-1}\partial_{s_i}$,
\[
  \mathrm{II}(e_i,e_j)=\Hess t(e_i,e_j)=-f^{-2}\Gamma^{t}_{ij}
  =\frac{f'(t)}{f(t)}\,\delta_{ij},
\]
so $\mathrm{II}$ is $(f'/f)$ times the identity and
$\Delta_g t=(n-1)f'/f=(n-1)\beta(t)$, giving \eqref{eq:busemann-alpha}
with $\alpha=(n-1)\beta$.
\end{proof}

The geometric content of Proposition~\ref{prop:dictionary} is the
interpretation of the horosphere-dependent drift coefficient as the
horospherical mean-curvature function in the Busemann setting;
Proposition~\ref{prop:concrete-realization} realizes a concrete class of
such coefficients by variable-curvature Cartan--Hadamard metrics.

Equation \eqref{eq:nonauto-main} itself is classical: with
$\rho(t):=\exp\int_0^t\alpha$ it is the Sturm--Liouville form
$(\rho U')'=\rho\,W'(U)$, and specific instances arise for invariant
solutions on $\Hy^n$: radial solutions give $\alpha(r)=(n-1)\coth r$,
and $H_{\mathrm h}$-invariant ones $\alpha(t)=(n-1)\tanh t$
\cite[\S2]{BirindelliMazzeo2009}. The inverse existence problem for
heterogeneous drifts satisfying the pinching bounds of
Proposition~\ref{prop:dictionary} is Problem~A in the Open problems
subsection below.

\begin{remark}[Geometry of the constant-drift hypothesis]
\label{rem:constant-drift-geometry}
The hypothesis $\Delta b\equiv c$ is restrictive but far from vacuous:
\begin{itemize}
\item \textbf{Hyperbolic space.} Every Busemann function on $\Hy^n$ has
  $\Delta b\equiv n-1$ (Example~\ref{ex:H2}); the same constant
  horospherical drift underlies the horospheric-wave propagation
  studied by Matano, Punzo, and Tesei~\cite{MatanoPunzoTesei2015}.
\item \textbf{Harmonic manifolds.} On a harmonic Cartan--Hadamard
  manifold the horospheres have constant mean curvature, so
  $\Delta b\equiv c$ for every Busemann function; the flat case
  $\mathcal M=\mathbb R^n$ is harmonic in this sense with $c=0$. Under
  strictly negative curvature, however, $c>0$: besides the rank-one
  symmetric spaces of noncompact type, this class contains the
  nonsymmetric Damek--Ricci spaces \cite{DamekRicci1992,Heber2006},
  which are harmonic and Einstein without being symmetric, and satisfy
  $\Delta b\equiv c>0$.
\end{itemize}
Note that constancy of the trace is much weaker than the affine
condition $\Hess b\equiv0$: since $\Delta b\equiv c>0$, the relevant
horospheres cannot be totally geodesic, so the affine splitting
mechanism of Remark~\ref{rem:orientation} does not apply to this
Busemann function. This does not, by itself, preclude some other
product splitting of $\mathcal M$, as illustrated by
$\Hy^2\times\mathbb R$.

The two main results therefore cover complementary regimes:
Proposition~\ref{prop:compatibility-front} applies when the mean
curvature is \emph{exactly constant} and $W$ is unbalanced at the
Fife--McLeod speed; Theorem~\ref{thm:main} applies when it is merely
\emph{uniformly positive} and the limits sit at equal levels of $W$.
\end{remark}

The dissipation identity \eqref{eq:integral-identity}, evaluated on a
Fife--McLeod front, gives
$\tau^*(W')\int_{\R}(U')^2=W(1)-W(-1)$, whence
\[
\operatorname{sign}\tau^*(W')=\operatorname{sign}\bigl(W(1)-W(-1)\bigr),
\]
with $\tau^*(W')=0$ exactly in the balanced case (for the same
bistable class of Proposition~\ref{prop:compatibility-front}).
Theorem~\ref{thm:main} and Proposition~\ref{prop:compatibility-front}
are thus two readings of the same dissipation identity: the former says
that a positive drift is incompatible with equal potential levels for a
nonconstant profile, while the latter states that a strictly increasing
Busemann front joining $-1$ to $1$ exists if and only if
$c=\tau^*(W')$.

Since $\tau^*$ is determined by the nonlinearity alone and is unique,
for any $c\neq\tau^*(W')$ there is no strictly increasing Busemann
front joining $-1$ to $1$ on a manifold with $\Delta_g b\equiv c$.

\section{Warped-line and rotational surfaces: existence and nonexistence
of symmetric layers}\label{sec:warped}

In this section we specialize to surfaces, where the ambient metric is
built from a smooth warping function, so, in contrast with
Remark~\ref{rem:coordinates}, all objects are smooth and the coordinate
form of the operator is exact by construction.

\subsection*{Standing assumptions on the potential}\label{subsec:standing}
Throughout, $W\in C^2(\R)$ is a balanced double-well potential in the
sense of Definition~\ref{def:balanced}, satisfying in addition
\begin{align}\label{eq:W-assumptions}
\begin{aligned}
  W>0 \ \text{on}\ (-1,1),\quad W(\pm1)=W'(\pm1)=0,\\
  W''(\pm1)>0,\quad W\ \text{even},\quad W'<0\ \text{on}\ (0,1).
\end{aligned}
\end{align}
The model case is $W(u)=\tfrac14(1-u^2)^2$. No sign or growth condition
is imposed outside $[-1,1]$; where a result requires the profile to
take values in $[-1,1]$, this is stated explicitly as a hypothesis.

\subsection*{The two symmetric classes}\label{subsec:classes}

\emph{(A) Warped-line surfaces.} $\mathcal M=\R^2_{(t,s)}$,
$g=dt^2+f(t)^2ds^2$, with $f\in C^\infty(\R)$, $f>0$, $f''\ge0$. Then
$K=-f''/f\le0$. Completeness follows from $g\ge dt^2$: for any curve
$\gamma(u)=(t(u),s(u))$ of length $L$ starting at $t(0)=t_0$, one has
$|t(u)-t_0|\le L_g(\gamma|_{[0,u]})\le L$ for every $u$, so the entire
curve, not merely its endpoint, stays in the slab $[t_0-L,t_0+L]$. On
that slab, the continuous positive function $f$ has a positive minimum
$m$, and $g\ge dt^2+m^2|ds|^2$ there controls the $s$-coordinate along
the curve as well. Hence a Cauchy sequence for $d_g$ is Cauchy in the
$t$- and, once $t$ is bounded, in the $s$-coordinates, and thus
converges in $(\mathcal M,g)$; by the Hopf--Rinow theorem
$(\mathcal M,g)$ is complete, hence Cartan--Hadamard.

For the leaves $\{t=\mathrm{const}\}$, the unit tangent field is
$e_s:=f(t)^{-1}\partial_s$. Since $\Gamma^t_{ss}=-f(t)f'(t)$, we have
$\mathrm{II}(e_s,e_s)=\Hess t(e_s,e_s)=f'(t)/f(t)$. As
$\Hess t(\partial_t,\cdot)=0$, it follows that
\[
\Delta_g t=\frac{f'(t)}{f(t)}=:a(t).
\]
Thus $a(t)$ is the mean curvature of the leaf with respect to
$\partial_t$. A leaf is a geodesic exactly when $f'(t)=0$; under
$f''\ge0$ the set $\{f'=0\}$ is an interval, possibly empty or a single
point. On $\Hy^2$: $f(t)=\cosh t$ gives the equidistants of a geodesic,
with a geodesic leaf at $t=0$; $f(t)=e^t$ gives the horocycles, with
$a\equiv1$ and no geodesic leaf (cf.\ Example~\ref{ex:H2}).

\emph{(B) Rotational surfaces.} $\mathcal M=\R^2$,
$g=dr^2+f(r)^2d\theta^2$, $\theta\in\R/2\pi\mathbb Z$, where $f$ is the
restriction to $[0,\infty)$ of a smooth \emph{odd} function on $\R$
with $f'(0)=1$, the standard polar smoothness condition ensuring that
$g$ extends smoothly across the pole, with $f>0$ on $(0,\infty)$ and
$f''\ge0$. Then $K=-f''/f\le0$, and since $f'$ is nondecreasing with
$f'(0)=1$ we get $f'\ge1$, hence $a=f'/f>0$: \emph{no leaf is a
geodesic}, consistently with the absence of closed geodesics on
Cartan--Hadamard manifolds. Note, however, that $f'\ge1$ and $f(0)=0$
force $f(r)\ge r$, so $a(r)\le f'(r)/r$; for many natural examples of
subexponential growth, including the Euclidean case $f(r)=r$, one has
$\inf a=0$, whereas for $f(r)=\sinh r$ one has $a(r)=\coth r$ and
$\inf a=1$. In class \textup{(B)} the uniform drift hypothesis of
Theorem~\ref{thm:main} therefore fails typically rather than
exceptionally (and Theorem~\ref{thm:main} does not apply in this class
regardless, since $r$ is not a Busemann function), which is why a
separate mechanism (Theorem~\ref{thm:radial}) is needed there.

In this section the coordinates $t$ and $r$ are geometric coordinates
of the warped models and are not assumed to be Busemann functions; the
reduced ODEs below hold independently of the Busemann setting. When
comparison with Theorem~\ref{thm:main} is made, it concerns the shared
dissipation mechanism, not an identification of these coordinates with
Busemann functions (cf.\ Remark~\ref{rem:no-drift-examples}).

In both classes a profile $u=U(t)$ (resp.\ $U(r)$) solves \eqref{eq:AC}
if and only if
\begin{equation}\label{eq:ode-warped}
  U''+a\,U'=W'(U),\qquad a=\frac{f'}{f},
\end{equation}
with $U'(0)=0$ in the rotational case, since smoothness of $u$ at the
pole forces the radial derivative to vanish there. As in
\eqref{eq:dissipation}, $H:=\tfrac12(U')^2-W(U)$ satisfies
$H'=-a\,(U')^2$.

\begin{remark}[Failure of uniform positive drift: what is and is not
ruled out]\label{rem:no-drift-examples}
When $a\ge c>0$ fails, the dissipation mechanism of
Theorem~\ref{thm:main} does not apply directly, and nonconstant
equal-level profiles may or may not exist, depending on finer features
of the geometry. Two contrasting instances illustrate this.

\begin{enumerate}
\item \textbf{Sign-changing drift in a non-Busemann coordinate:
existence.} In class \textup{(A)} take $f(t)=\cosh t$, so
$a(t)=\tanh t$ vanishes at $t=0$ and changes sign. Here $t$ is the
signed-distance coordinate to a complete geodesic of $\Hy^2$, not a
Busemann function (indeed $\Delta t=\tanh t$ changes sign, whereas
every Busemann function satisfies $\Delta b\ge0$ by convexity); thus
neither Theorem~\ref{thm:main} nor Theorem~\ref{thm:translational}
applies, the former because $t$ is not Busemann, the latter because
$a$ is not nonnegative on $\mathbb R$. A nonconstant equal-level
profile does exist here: Theorem~\ref{thm:existence} provides an odd,
strictly increasing solution of $U''+\tanh(t)\,U'=W'(U)$ with
$U(\pm\infty)=\pm1$, whose nodal set is the geodesic leaf $\{t=0\}$;
compare \cite[Prop.~2.4]{BirindelliMazzeo2009}. Positivity of the
drift is thus essential to the dissipation mechanism, but this example
does not constitute a sharpness test within the Busemann class of
Theorem~\ref{thm:main}.

\item \textbf{Drift vanishing at infinity: nonexistence persists, by a
different mechanism.} In class \textup{(B)} take $f(r)=r$, the
Euclidean plane, so $a(r)=1/r>0$ but $\inf a=0$ and
Theorem~\ref{thm:main} does not apply. Nevertheless
Theorem~\ref{thm:radial} still rules out nonconstant radial profiles
with values in $[-1,1]$ and limit in $\{\pm1\}$: the condition
$U'(0)=0$ at the pole pins the Hamiltonian at
$H(0)=-W(U(0))\le0$, and monotonicity of $H$ together with
$H(\infty)=0$ forces $H\equiv0$. The pole substitutes for global
uniform dissipation.

This concerns \emph{radial} profiles only. Nonconstant entire
solutions of \eqref{eq:AC} on $\R^2$ certainly exist, the planar layer
$u(x_1,x_2)=U_0(x_1)$, with $U_0$ the classical heteroclinic of
$U_0''=W'(U_0)$, but they are not radial, and so lie outside the
profile ansatz used here.
\end{enumerate}

The horocyclic case $f(t)=e^t$, $a\equiv1$, is \emph{not} an instance
of failing drift: there $\inf a=1>0$ and Theorem~\ref{thm:main}
applies, so no equal-level profile exists
(Corollary~\ref{cor:horocyclic}). The nonconstant solution of
$U''+U'=W'(U)$ in \cite[Prop.~2.2]{BirindelliMazzeo2009} connects
\emph{unequal} potential levels; see Remark~\ref{rem:sharpness}.
\end{remark}

\begin{theorem}[Nonexistence: monotone warpings]
\label{thm:translational}
In class \textup{(A)}, assume $a=f'/f\ge0$ with $a\not\equiv0$
(equivalently, $f$ nondecreasing and nonconstant). Let $U\in C^2(\R)$
solve \eqref{eq:ode-warped} with finite limits $\ell_\pm$ at
$\pm\infty$ satisfying $W(\ell_-)=W(\ell_+)$. Then $U$ is constant. In
particular, no entire solution of \eqref{eq:AC} depending only on $t$
(i.e.\ $\partial_s$-invariant) on such a surface connects the two
wells $\pm1$ across the foliation. The case $a\le0$ reduces to this
one via $t\mapsto-t$, which exchanges $\ell_-$ and $\ell_+$; the
balance hypothesis being symmetric, the conclusion is unchanged.
\end{theorem}
\begin{proof}
Exactly as in the proof of Theorem~\ref{thm:main}, with $\alpha=a\ge0$
in place of $\alpha\ge c>0$: $H$ is nonincreasing, the sequential
argument identifies $H(\pm\infty)=-W(\ell_\pm)$, and the balance
hypothesis forces $H$ constant, hence
\[
  a(t)\,U'(t)^2\equiv 0 \qquad (t\in\R).
\]

Since $a\not\equiv0$ and $a=f'/f$ is continuous, there is an open
interval $I$ on which $a>0$; hence $U'\equiv0$ on $I$, so $U\equiv q$
there and \eqref{eq:ode-warped} gives $W'(q)=0$. Fix $t_1\in I$. Then
$(U,U')(t_1)=(q,0)$ and the constant $q$ solves \eqref{eq:ode-warped}
with the same datum; as $W\in C^2$ by \eqref{eq:W-assumptions}, the
field $(U,V)\mapsto(V,\,W'(U)-a(t)V)$ is locally Lipschitz in $(U,V)$,
so uniqueness for the Cauchy problem gives $U\equiv q$ on $\mathbb R$.
\end{proof}

Neither the uniform bound $\inf a>0$ nor the structure of the zero set
of $a$ is needed: the essential ingredient beyond the $C^1$ setting of
Theorem~\ref{thm:main} is Cauchy uniqueness, exactly as in
Corollary~\ref{cor:nonnegative-drift}, and the convexity of $f$ plays
no role in this step.

For constant $f$ the drift vanishes identically and
\eqref{eq:ode-warped} becomes $U''=W'(U)$, which admits the classical
heteroclinic joining $\pm1$ for any balanced double well. Some
nonvanishing of the drift is therefore indispensable: $a\not\equiv0$
here, $\Delta_g b\not\equiv0$ in Corollary~\ref{cor:nonnegative-drift},
and $\inf_{\mathcal M}\Delta_g b>0$ in Theorem~\ref{thm:main}, where no
local Lipschitz condition on $W'$ is imposed.

\begin{corollary}[No horocyclic two-well profiles on $\Hy^2$]
\label{cor:horocyclic}
The equation $U''+U'=W'(U)$ admits no nonconstant solution with finite
limits at equal potential levels; in particular, no bounded entire
solution of \eqref{eq:AC} on $\Hy^2$ invariant under a parabolic
one-parameter group of isometries connects the two wells $\pm1$. This
is the case $f(t)=e^t$ of Theorem~\ref{thm:translational}, the
constant-drift case of Theorem~\ref{thm:main}, and recovers for
$\Hy^2$ the parabolic nonexistence statement of
\cite[Prop.~2.2, first assertion]{BirindelliMazzeo2009}.
\end{corollary}

\begin{theorem}[Nonexistence: radial profiles]\label{thm:radial}
In class \textup{(B)}, let $u=U(r)$ be a radial solution of
\eqref{eq:AC} with values in $[-1,1]$ and
$\lim_{r\to\infty}U(r)\in\{-1,+1\}$. Then $U$ is constant, equal to
$+1$ or $-1$.
\end{theorem}
\begin{proof}
Smoothness of $u$ at the pole gives $U'(0)=0$, so
$H(0)=-W(U(0))\le0$ by \eqref{eq:W-assumptions} and $|U(0)|\le1$.
Since $a>0$ on $(0,\infty)$, the dissipation identity $H'=-a(U')^2$
makes $H$ nonincreasing there.

To identify $H(\infty)$, argue as in Theorem~\ref{thm:main}: writing
$\ell:=U(\infty)\in\{\pm1\}$, the mean value theorem on $[m,m+1]$
gives $\tau_m\to\infty$ with $U'(\tau_m)\to0$, so
$H(\tau_m)\to-W(\ell)=0$; monotonicity of $H$ then forces
$H(\infty)=0$.

Hence $H(r)\ge0$ for every $r>0$, and since $H$ is continuous at the
pole, $H(0)\ge0$. With $H(0)\le0$ this gives $H(0)=0$, and
monotonicity squeezes $H\equiv0$ on $[0,\infty)$. Then $H'\equiv0$,
which by $a>0$ forces $U'\equiv0$, so $U\equiv U(0)$. Finally
$W(U(0))=-H(0)=0$, and the only zeros of $W$ in $[-1,1]$ are $\pm1$.
\end{proof}

The mechanism here differs from that of Theorem~\ref{thm:main}: no
uniform lower bound on $a$ is needed, only $a(r)>0$ for $r>0$ together
with the smoothness condition $U'(0)=0$, which fixes the Hamiltonian
at the pole and, by monotonicity, forces $H\equiv0$, a boundary
condition at a single point substitutes for global uniform
dissipation. The argument in fact only needs $W\ge0$ on the range of
$U$ and $U(\infty)$ a zero of $W$; we have stated it in the
double-well setting for consistency with this section. This is the
same elementary energy mechanism used in the hyperbolic radial case by
Birindelli--Mazzeo \cite[Prop.~2.1]{BirindelliMazzeo2009};
Theorem~\ref{thm:radial} isolates the two features actually needed,
$a>0$ on $(0,\infty)$ and the pole condition $U'(0)=0$, and thereby
covers every rotational surface of class~\textup{(B)}, where
$\inf a=0$ is the typical situation.

\begin{theorem}[Existence: the layer across a geodesic leaf,
reflection-symmetric case]\label{thm:existence}
In class \textup{(A)}, assume in addition that $f$ is \emph{even}, so
that $a=f'/f$ is odd and the leaf $\{t=0\}$ is a geodesic; together
with the convexity $f''\ge0$ of class~\textup{(A)}, this gives $a\ge0$
on $[0,\infty)$. Then \eqref{eq:ode-warped} admits an odd solution
$U\in C^3(\R)$ with
\[
  U(\pm\infty)=\pm1,\qquad |U|<1,\qquad U'>0\ \text{on}\ \R ,
\]
and $U\in C^\infty$ if $W\in C^\infty$. Consequently $u(t,s)=U(t)$ is
a bounded, nonconstant entire solution of \eqref{eq:AC} on $(\M,g)$
whose zero set is exactly the geodesic leaf $\{t=0\}$.
\end{theorem}
\begin{proof}
For $\sigma>0$ let $U_\sigma$ be the maximal solution of
\eqref{eq:ode-warped} on $[0,\infty)$ with $U_\sigma(0)=0$,
$U_\sigma'(0)=\sigma$. We first record \emph{global existence while
$U_\sigma\in[0,1]$}: on any interval where $U_\sigma\in[0,1]$, $H$ is
nonincreasing (as $a\ge0$), and $0\le W(U_\sigma)\le W(0)$ by
\eqref{eq:W-assumptions}, so
$\tfrac12(U_\sigma')^2=H+W(U_\sigma)\le H(0)+W(0)=\tfrac{\sigma^2}2$;
thus $(U_\sigma,U_\sigma')$ stays bounded. Since the vector field
$(U,V)\mapsto(V,\,W'(U)-a(t)V)$ is locally Lipschitz in $(U,V)$ and
$a$ is continuous, boundedness on compact intervals prevents
finite-time blow-up; hence $U_\sigma$ extends for as long as
$U_\sigma\in[0,1]$, up to the first exit time from this strip. Define
\begin{align*}
  \mathcal A&:=\{\sigma>0:\ \exists\,t_1>0,\ U_\sigma(t_1)=1,\
  U_\sigma'>0 \text{ on }[0,t_1]\},\\
  \mathcal B&:=\{\sigma>0:\ \exists\,t_1>0,\ U_\sigma'(t_1)=0,\
  U_\sigma\in(0,1) \text{ on }(0,t_1],\ U_\sigma'>0 \text{ on }(0,t_1)\}.
\end{align*}
Thus $\mathcal A$ collects the slopes whose profile reaches the level
$1$ with positive speed, and $\mathcal B$ those whose profile turns
around strictly below it.

\emph{Structure of the argument (shooting method).} We show below
that $\mathcal A$ and $\mathcal B$ are open, disjoint and nonempty.
Since $(0,\infty)$ is connected, their union cannot exhaust it, so
there is a critical slope $\sigma^*$ in neither; its profile is the
desired heteroclinic on $[0,\infty)$, and the odd extension completes
the proof.

\emph{Disjoint and open.} Disjointness: whichever event occurs first
excludes the other; a simultaneous event $U_\sigma(t_1)=1$,
$U_\sigma'(t_1)=0$ would force, by uniqueness at the equilibrium
$(1,0)$ (note $W'(1)=0$), $U_\sigma\equiv1$, contradicting
$U_\sigma(0)=0$. Openness: both defining events are stable under
small perturbations of $\sigma$, by continuous dependence of
solutions of \eqref{eq:ode-warped} on the initial datum: in
$\mathcal A$, the level $1$ is crossed with $U_\sigma'(t_1)>0$; in
$\mathcal B$, the first critical point satisfies
$U_\sigma''(t_1)=W'(U_\sigma(t_1))<0$ strictly, by
\eqref{eq:W-assumptions}; in both cases nearby solutions exhibit the
same first event.

\emph{$\mathcal B$ contains all $\sigma$ with $\sigma^2<2W(0)$.} For
such $\sigma$, $H(0)=\tfrac{\sigma^2}2-W(0)<0$, and $H$ is
nonincreasing, so $H<0$ for all $t$ while $U_\sigma\in[0,1]$. The
level $1$ cannot be reached: there
$\tfrac12(U_\sigma')^2=H+W(1)=H<0$, impossible; moreover
$\tfrac12(U_\sigma')^2=H+W(U_\sigma)<W(U_\sigma)$ requires
$W(U_\sigma)\ge-H\ge-H(0)>0$, keeping $U_\sigma$ bounded away from
$1$ (where $W\to0$). Suppose, for contradiction, that $U_\sigma'>0$
for all $t>0$; then $U_\sigma\nearrow L$ with $L\in(0,1)$ by the
above. Since $a\ge0$ and $U_\sigma'>0$,
\[
  U_\sigma''=W'(U_\sigma)-a\,U_\sigma'\;\le\;W'(U_\sigma)
  \;\longrightarrow\;W'(L)<0 ,
\]
so $U_\sigma''\le\tfrac12W'(L)<0$ for all large $t$ (the limit being
strictly negative), forcing $U_\sigma'$ to reach $0$ in finite time,
a contradiction. Hence $U_\sigma'$ vanishes at a first time $t_1$,
with $U_\sigma(t_1)\in(0,1)$ ($U_\sigma$ increased from $0$ and
stayed below $1$): $\sigma\in\mathcal B$.

\emph{$\mathcal A$ contains all $\sigma>\sigma_0:=
\max\{\sqrt2,\ \sqrt{8W(0)},\ 8\sqrt2\,M_a\}$, where
$M_a:=\max_{t\in[0,1]}a(t)\ge0$} (finite and well-defined because
$a\in C^\infty$ on $\R$). Let $T$ be the first time at which either
$U_\sigma=1$ or $H=\sigma^2/4$. On $[0,T)$ we have $H\ge\sigma^2/4$
and $U_\sigma\in[0,1)$, so, using $W(U_\sigma)\ge0$ on $[0,1]$ by
\eqref{eq:W-assumptions},
\[
  \tfrac12(U_\sigma')^2=H+W(U_\sigma)\ \ge\ \tfrac{\sigma^2}4
  \quad\Longrightarrow\quad U_\sigma'\ \ge\ \tfrac{\sigma}{\sqrt2},
\]
and also, since $H\le H(0)$ and $W(U_\sigma)\le W(0)$ (as $W$ is
nonincreasing on $[0,1]$ by \eqref{eq:W-assumptions}),
$U_\sigma'\le\sigma$. From $U_\sigma'\ge\sigma/\sqrt2$ and
$U_\sigma(0)=0$, the level $1$ is reached by time $\sqrt2/\sigma$
unless the event $H=\sigma^2/4$ occurs first; in either case
\[
  T\ \le\ \frac{\sqrt2}{\sigma}\ <\ 1
  \qquad(\sigma>\sqrt2).
\]
Suppose the event at $T$ is $H(T)=\sigma^2/4$ with $U_\sigma(T)<1$.
The dissipated energy satisfies
\[
  \frac{\sigma^2}4-W(0)
  \;=\;H(0)-H(T)
  \;=\;\int_0^T a\,(U_\sigma')^2
  \;\le\;M_a\,\sigma^2\,T
  \;\le\;\sqrt2\,M_a\,\sigma ,
\]
using $T<1$ (so $[0,T]\subset[0,1]$ and hence $a\le M_a$ on $[0,T]$)
and $U_\sigma'\le\sigma$. For $\sigma^2\ge8W(0)$ the left-hand side
is at least $\sigma^2/8$, so $\sigma\le8\sqrt2\,M_a$, contradicting
$\sigma>\sigma_0$ (when $M_a=0$, the inequality
$\sigma^2/4\le W(0)$ contradicts $\sigma^2\ge8W(0)>0$). Hence the
first event is $U_\sigma=1$, with $U_\sigma'\ge\sigma/\sqrt2>0$
throughout: $\sigma\in\mathcal A$.

\emph{The critical parameter.} Since $(0,\infty)$ is connected and
$\mathcal A,\mathcal B$ are open, disjoint, and nonempty, there is
$\sigma^*\in(0,\infty)\setminus(\mathcal A\cup\mathcal B)$. Write
$U:=U_{\sigma^*}$ and
\[
  t^*:=\sup\{T>0:\ U'>0 \text{ on }(0,T] \text{ and } U<1 \text{ on }[0,T]\}
  \;\in\;(0,\infty],
\]
positive since $U'(0)=\sigma^*>0$; the solution exists up to $t^*$ by
the global-existence observation above. If $t^*<\infty$, then at
$t^*$ either $U'(t^*)=0$ or $U(t^*)=1$ (note $U(t^*)=0$ is
impossible, as $U$ is increasing on $(0,t^*)$ from $U(0)=0$). The
case $U'(t^*)=0$, $U(t^*)\in(0,1)$ puts $\sigma^*\in\mathcal B$; the
case $U(t^*)=1$, $U'(t^*)>0$ puts $\sigma^*\in\mathcal A$; the case
$U(t^*)=1$, $U'(t^*)=0$ forces $U\equiv1$ by uniqueness at the
equilibrium $(1,0)$, contradicting $U(0)=0$. All three are excluded,
so $t^*=\infty$: $U'>0$ and $0<U<1$ on $(0,\infty)$, and
$U\nearrow L\le1$. The value $L\in(0,1)$ is excluded exactly as in
the $\mathcal B$-step ($U''\le\tfrac12W'(L)<0$ for large $t$
contradicts $U'>0$ globally); hence $L=1$.

\emph{Full decay of the derivative.} By the mean value theorem there
are $\tau_m\to\infty$ with $U'(\tau_m)\to0$, so
$H(\tau_m)=\tfrac12U'(\tau_m)^2-W(U(\tau_m))\to-W(1)=0$; since $H$ is
monotone on $[0,\infty)$, $H(t)\to0$, and therefore
$\tfrac12U'(t)^2=H(t)+W(U(t))\to0$, i.e.\ $U'(t)\to0$ as
$t\to\infty$.

\emph{Odd extension and regularity.} Since $a$ is odd and $W'$ is odd
(\eqref{eq:W-assumptions}), the function $\widetilde U(t):=-U(-t)$
solves \eqref{eq:ode-warped} with the same initial data
$(0,\sigma^*)$; by uniqueness, the solution of \eqref{eq:ode-warped}
on $\R$ with data $(0,\sigma^*)$ is odd and coincides with $U$ on
$[0,\infty)$. The extension satisfies $U(\pm\infty)=\pm1$, $|U|<1$
(strict on $(0,\infty)$ as shown, on $(-\infty,0)$ by oddness, and
$U(0)=0$), and $U'>0$ ($U'(0)=\sigma^*$; $U'(t)=U'(-t)$ by oddness).
Standard ODE regularity (bootstrapping) then gives $U\in C^3$, and
$U\in C^\infty$ if $W\in C^\infty$.
\end{proof}

No curvature pinching and no bound on $a$ are required: the only
hypotheses are those of class~\textup{(A)} together with the
reflection symmetry of $f$. For $f(t)=\cosh t$ the theorem gives the
layer across a complete geodesic of $\Hy^2$; a monotone solution of
the same ODE is obtained variationally in
\cite[Prop.~2.4]{BirindelliMazzeo2009}, whose proof isolates
$\int_{-1}^1 W'=0$ as the additional property needed beyond the
standing double-well hypotheses; oddness here follows instead from
the symmetry of $a$ and $W'$ together with uniqueness for the Cauchy
problem.

\begin{remark}[Stability of the symmetric layer]
\label{rem:stability}
Let $u(t,s)=U(t)$ be the layer of Theorem~\ref{thm:existence} and
$J:=\Delta_g-W''(u)$ its linearized operator. Differentiating
$U''+aU'=W'(U)$ in $t$ gives the exact identity
\begin{equation}\label{eq:stability-identity}
J[U'] = -a'(t)\,U'(t), \qquad a' = -K-a^2,
\end{equation}
where $K=-f''/f$ is the Gauss curvature. Hence $a'\ge0$ on $\R$ is
equivalent to $K\le-a^2$, and in that case $v:=U'>0$ is a positive
supersolution of $L:=-J=-\Delta_g+W''(u)$, i.e.\ $Lv\ge0$. For
$\varphi\in C_c^\infty(\mathcal M)$, write $\varphi=v\psi$; the
ground-state substitution gives, for every such $\varphi$,
\begin{align*}
Q(\varphi)&:=\int_{\mathcal M}\bigl(|\nabla\varphi|^2+W''(u)\varphi^2\bigr)\,d\mu_g=\int_{\mathcal M}v^2|\nabla\psi|^2\,d\mu_g+\int_{\mathcal M}\frac{Lv}{v}\,\varphi^2\,d\mu_g
\;\ge\;0,
\end{align*}
the standard Allegretto--Piepenbrink mechanism (cf.\
\cite{FischerColbrieSchoen1980,PigolaRigoliSetti2008} for the broader
context); the argument is self-contained and does not require the
supersolution to lie in $L^2$. This gives nonnegativity of $Q$; an
$L^2$-normalized ground state would require decay estimates not
pursued here.

Geometrically, $K\le-a^2$ says the ambient curvature is at least as
negative as the squared mean curvature of the leaves; it holds
automatically when $K\le-\sup a^2$, and in particular for
$f(t)=\cosh t$, where $a=\tanh t$ and $a'=\operatorname{sech}^2t>0$.
When $a'$ changes sign the preceding argument gives no information.
The corresponding stability and Morse-index questions, for the
warping $f(t)=1+t^2$, are discussed in the Open problems subsection
below. This remark concerns class \textup{(A)}; the extension to
general Busemann profiles is not treated here.
\end{remark}

Within the warped-line class the results above establish: if $f$ is
nondecreasing and nonconstant, no equal-level profile exists
(Theorem~\ref{thm:translational}); this does not require the absence
of geodesic leaves, since $f'$ may vanish on a half-line, and every
leaf with $f'(t)=0$ is then geodesic. If $f$ is even and convex, the
central leaf $\{t=0\}$ is geodesic and reflection-symmetric, and
Theorem~\ref{thm:existence} produces an odd, strictly monotone
two-well layer with nodal set that leaf. In the rotational class,
whose leaves are never geodesics, there are no radial two-well
profiles (Theorem~\ref{thm:radial}).

These are partial results, not an ``if and only if'' criterion; the
case of a convex warping with an isolated interior minimum but no
reflection symmetry is recorded in the Open problems subsection below.

\section{Concluding remarks and open problems}\label{sec:remarks}

\begin{remark}[Energy of the symmetric layers]\label{rem:energy}
For the layers of Theorem~\ref{thm:existence}, the natural reduced
functional is the \emph{energy per unit orbit parameter},
\[
\mathcal E[U]
:=
\int_{\mathbb R}
\left(\tfrac12(U')^2+W(U)\right)f(t)\,dt,
\]
obtained by formally factoring the $s$-integration out of the
Allen--Cahn energy. Since the $s$-orbits have infinite parameter
length and the reduced integrand is nonnegative and not identically
zero, the total energy of the nonconstant layer $u(t,s)=U(t)$ on
$\mathcal M$ is infinite.

The quantity $\mathcal E[U]$ is not an energy per unit leaf length in
any uniform sense, since the length element on the leaf
$\{t=\mathrm{const}\}$ is $f(t)\,ds$ and therefore varies with $t$. We
make no claim that $\mathcal E[U]<\infty$ in general; this depends on
quantitative decay of $1-|U|$ and $U'$ relative to the growth of $f$,
estimates that are not needed here.
\end{remark}

\subsection*{Open problems}\label{subsec:open}
\begin{enumerate}

\item \emph{Beyond the ansatz.} Classify the bounded entire solutions
of \eqref{eq:AC} on a Cartan--Hadamard manifold with $K\le-\kappa_2<0$
whose asymptotic behaviour is prescribed through a Busemann function,
say $u\to\ell_\pm$ as $b\to\pm\infty$, uniformly on horospheres in the
sense $\sup_{x\in\Sigma_t}|u(x)-\ell_\pm|\to0$, \emph{without} assuming
$u=U\circ b$. This would be the variable-curvature analogue of the
\emph{symmetry} theorems of Birindelli--Mazzeo
\cite{BirindelliMazzeo2009}, which force certain solutions on $\Hy^n$
with invariant asymptotic data to be invariant, and which rely on the
isometry group, unavailable here. The convex-barrier technique of
\cite{CavalcanteEspinarMarin2026a}, developed for the asymptotic
Dirichlet problem and likewise independent of any isometry group, may
offer a route to such a classification; the precise relation between
their boundary data and the horospherical limits $\ell_\pm$ used here
remains to be worked out. Theorem~\ref{thm:main} and
Remark~\ref{rem:consistency} suggest, but do not prove, that under
equal-level data such solutions should be constant.

\item \emph{Asymmetric geodesic leaves.} Existence of a two-well layer
across the geodesic leaf of a convex warping with an isolated
interior minimum but without evenness. The natural route is
constrained minimization of $\mathcal E$ with a pinning mechanism
confining minimizing transitions near the minimum of $f$; this is the
subject of work in preparation and is not resolved by the shooting
method of Theorem~\ref{thm:existence}, which relies on reflection
symmetry to collapse two one-sided shooting parameters into one.

\item \emph{Manifolds with constant horospherical mean curvature.}
Classify the Cartan--Hadamard manifolds for which $\Delta b\equiv c$
for some, or every, Busemann function, and determine the possible
values of $c$ in terms of the curvature and the geometry of the ideal
boundary. This is the geometric question underlying the compatibility
condition of Proposition~\ref{prop:compatibility-front}: for fixed
$W$, well-to-well fronts exist exactly when $c=\tau^*(W')$. The
hypothesis holds for hyperbolic spaces, harmonic manifolds, and
Damek--Ricci spaces (Remark~\ref{rem:constant-drift-geometry}), and a
number of classification results are known in the homogeneous
setting; a classification under the general hypotheses of the present
paper, without homogeneity, and for a single ideal point rather than
every Busemann function, is beyond its scope.

\item \emph{Morse index when $a'$ changes sign.} For warpings such as
$f(t)=1+t^2$ the condition $a'\ge0$ fails, since
$a'(t)=2(1-t^2)/(1+t^2)^2$, and the identity $J[U']=-a'U'$ no longer
supplies a positive supersolution with a fixed sign. Although
$K(t)=-2/(1+t^2)\to0$, this alone does not determine the asymptotic
spectral behaviour of $J$: in the global warped coordinates,
\[
J=\partial_t^2+\frac{2t}{1+t^2}\partial_t
+\frac{1}{(1+t^2)^2}\partial_s^2-W''(U(t)),
\]
whose coefficients do not converge to those of the Euclidean layer
operator $\partial_t^2+\partial_s^2-W''(\pm1)$, the tangential
coefficient $1/(1+t^2)^2$ vanishes rather than tending to $1$.
Determining stability or Morse index in this regime requires a
separate spectral analysis, which we do not pursue here.

\item \emph{Problem A (unbalanced layers with variable friction).}
\label{item:problem-A-prime}
In the regime of Proposition~\ref{prop:dictionary},
Theorem~\ref{thm:main} precludes nonconstant profiles with
equal-level limits, so the heterogeneous regime has content only for
unbalanced data, $W(\ell_-)<W(\ell_+)$, where the strict climbing
inequality leaves room for nonconstant profiles. The equation is
\[
U''+\alpha(t)\,U'=W'(U),\qquad
\alpha(\R)\subset\bigl[(n-1)\sqrt{\kappa_2},\,(n-1)\sqrt{\kappa_1}\bigr].
\]
In the autonomous case $\alpha\equiv c$, uniqueness of the
Fife--McLeod speed confines existence to the single value
$c=\tau^*(W')$ (Proposition~\ref{prop:compatibility-front}). The
classical nonautonomous reductions on $\Hy^n$, $\alpha(r)=(n-1)\coth r$
for radial solutions and $\alpha(t)=(n-1)\tanh t$ for
$H_{\mathrm h}$-invariant ones \cite[\S2]{BirindelliMazzeo2009},
illustrate the general mechanism of \eqref{eq:nonauto-main} but do
\emph{not} lie in the bounded regime above: $\coth r\to\infty$ as
$r\to0^+$, so $\alpha$ violates the upper bound
$(n-1)\sqrt{\kappa_1}$, while $\tanh t\to0$ at $t=0$, so $\alpha$
violates the lower bound $(n-1)\sqrt{\kappa_2}>0$. We do not address
the structure of the admissible set of bounded drift profiles
$\alpha$; determining which such profiles admit a monotone
heteroclinic joining the wells is a natural open direction.
\end{enumerate}

\subsection*{Acknowledgements}

The authors thank Professor Stefano Nardulli for motivating and
encouraging them to work on this problem. The first ideas for this
work emerged from his invitation to participate, at UFABC, Brazil, in
July 2023, in activities associated with FAPESP project
No.~2021/05256-0, ``Problemas variacionais geom\'etricos: exist\^encia,
regularidade e caracteriza\c{c}\~ao geom\'etrica de solu\c{c}\~oes''.

\end{document}